\documentclass[11pt]{amsart}

\usepackage{amsthm}
\usepackage{amssymb}
\usepackage{amsmath}
\usepackage{url}                     
\usepackage{fullpage}
\usepackage{tikz-cd}                 
\usepackage{tikz}
\usepackage{enumitem}
\usepackage{mathrsfs}
\usepackage{mathtools}

\makeatletter
\def\l@subsection{\@tocline{2}{0pt}{2.5pc}{5pc}{}}
\makeatother

\usepackage[                         
   backend=biber,
   style=numeric,
   sortlocale=de_DE,
   url=false,
   doi=true,
   eprint=false
]{biblatex}

\usepackage{hyperref}

\newtheorem{theorem}{Theorem}[section]
\newtheorem{lemma}[theorem]{Lemma}

\newtheorem{proposition}[theorem]{Proposition}

\theoremstyle{definition}
\newtheorem{definition}[theorem]{Definition}
\newtheorem{remark}[theorem]{Remark}

\newcommand{\C}{\mathbb{C}}
\newcommand{\R}{\mathbb{R}}

\newcommand{\Z}{\mathbb{Z}}
\newcommand{\N}{\mathbb{N}}

\newcommand{\fg}{\mathfrak{g}}

\newcommand{\MC}[1]{\mathcal{#1}}

\newcommand{\eps}{\epsilon}

\newcommand{\inj}{\hookrightarrow}
\newcommand{\ra}{\rightarrow}
\newcommand{\om}[1]{\operatorname{#1}}
\newcommand{\ovl}[1]{\overline{#1}}
\newcommand{\wt}[1]{\widetilde{#1}}

\numberwithin{equation}{section}

\begin{document}

\date{\today}
\author{Gard Olav Helle}
\title{Singular Quiver Varieties over Extended Dynkin Quivers}

\begin{abstract}
We classify the singularities in the unframed Nakajima quiver varieties
associated with extended Dynkin quivers
and the corresponding minimal imaginary root with a small restriction on the
parameter and use this to construct a number of hyper-K\"{a}hler bordisms between binary polyhedral spaces. 
\end{abstract}

\maketitle
\tableofcontents


\section{Introduction}
In \cite{Nakajima94} Nakajima introduced a family of spaces he called
quiver varieties. A quiver is simply a finite directed graph $(Q,I)$ where $I$ is the set of vertices and $Q$ is the set of edges. We typically denote the quiver by $Q$. Given a dimension vector $v\in \Z^I_{\geq 0}$, we form the
vector space 
\[ \om{Rep}(Q,v) \coloneqq  \bigoplus_{(h\colon i\to j)\in Q} \om{Hom}_\C(\C^{v_i},\C^{v_j}), \]
which carries a natural linear action of the compact Lie group $G_v \coloneqq \prod_{i\in I} U(v_i)$. The doubled quiver $\ovl{Q}$ is obtained from $Q$ by adjoining an opposite edge $\ovl{h}\colon j\ra i$ for each edge $h:i\to j$ in $Q$. In this situation one may
give the complex vector space $\om{Rep}(\ovl{Q},v)$ a natural quaternionic structure preserved by the action of $G_v$. There is an associated 
hyper-K\"{a}hler moment map
$\mu\colon \om{Rep}(\ovl{Q},v)\ra \R^3\otimes \fg_v$, where $\fg_v=\om{Lie}(G_v)$. The quiver varieties associated with $Q$ and $v$ are then defined to
be the hyper-K\"{a}hler quotients
\[ \MC{M}_\xi(Q,v) \coloneqq \mu^{-1}(\xi)/G_v  \]
for $\xi = (\xi_1,\xi_2,\xi_3)\in \R^3\otimes \R^I$. Here, $\xi$ is regarded
as an element of $\R^3\otimes \fg_v$ using a canonical linear map from
$\R^I$ onto the center of the Lie algebra. Given $w\in \Z^I$ let
\[ D_w = \{\zeta\in \R^I: \zeta\cdot w = \sum_i \zeta_iw_i=0\}\subset \R^I. \]
It is then necessary that $\xi\in \R^3\otimes D_v$ for $\MC{M}_\xi(Q,v)$
to be nonempty, however, for almost all such parameters
the quiver variety $\MC{M}_\xi(Q,v)$ carries the structure of a smooth
hyper-K\"{a}hler manifold. More generally, there is a decomposition
\[ \MC{M}_\xi(Q,v) = \MC{M}_\xi^{\om{reg}}(Q,v)\cup 
\MC{M}_\xi^{\om{sing}}(Q,v),  \]
where the regular set $\MC{M}_\xi^{\om{reg}}(Q,v)$ is open and carries the
structure of a smooth hyper-K\"{a}hler manifold, while the singular set
$\MC{M}_\xi^{\om{sing}}(Q,v)$ is its closed complement.   

An extended Dynkin quiver $Q$ is a quiver whose underlying unoriented
graph is an extended Dynkin diagram of type $\wt{ADE}$, that is, type 
$\wt{A}_n$, $\wt{D}_n$, $\wt{E}_6$, $\wt{E}_7$ or $\wt{E}_8$. In this
situation there is a distinguished dimension vector $\delta\in \Z^I_{\geq 0}$; the minimal positive imaginary root in the associated root system. 
The purpose of this paper is to study the singular members of the family of quiver varieties $\MC{M}_\xi(Q,\delta)$ when $Q$ is an extended Dynkin quiver. This family of spaces, whose non-singular members are the so-called
ALE spaces, was first constructed and studied by Kronheimer \cite{Kronheimer89} in a slightly different form. The fact that Kronheimer's
construction can be expressed in the above form is explained in
\cite[p.~372-373]{Nakajima94}. 

The McKay correspondence \cite{McKay} sets up a bijection between the isomorphism classes of finite subgroups $\Gamma\subset \om{SU}(2)$ and the extended Dynkin diagrams of type $\wt{ADE}$. Kronheimer exploited this correspondence to show that the (non-empty) non-singular members
of the family $\MC{M}_\xi(Q,\delta)$ for $\xi\in \R^3\otimes D_\delta$ 
are smooth $4$-dimensional hyper-K\"{a}hler manifolds diffeomorphic to the minimal resolution of the quotient singularity $\C^2/\Gamma$ where $\Gamma\subset \om{SU}(2)$ is the finite subgroup associated with the underlying graph of $Q$ under the McKay correspondence. 
    
To state our first main result let $Q$ be an extended Dynkin quiver with
vertex set $I$ and minimal positive imaginary root $\delta\in \Z^I$. By deleting any vertex $i\in I$ with $\delta_i=1$ from $Q$ one recovers the associated (non-extended) Dynkin graph of type $ADE$. Identify the set of vertices with $\{0,1,\cdots,n\}$ for some $n\in \N$ such that $\delta_0=1$. 
One may then realize the root system associated with the underlying Dynkin graph as a subset $\Phi\subset \Z^n\subset \R^n$ with the coordinate vectors
as a set of simple roots. Furthermore, there is a natural way to
identify $\R^n\cong D_\delta \subset \R^{n+1}$ thereby identifying the
set of parameters $\R^3\otimes D_\delta \cong \R^3\otimes \R^n$. With this
in mind, our first main result can be stated as follows. 

\begin{theorem} \label{Singularity-Theorem}
Let $\xi = (\xi_1,\xi_2,\xi_3)\in \R^3\otimes \R^n$
satisfy $\xi_1=0$. Then if
$\Phi\cap \xi^\perp = \{\alpha\in \Phi : \alpha\cdot \xi_2=\alpha \cdot \xi_3=0 \}$ is nonempty, it is a root system in the subspace it spans and admits
a decomposition into root systems of type $ADE$:
\begin{equation} \label{Intro-RootSpaceDecomp}
 \Phi\cap \xi^\perp = \Phi_1\cup \Phi_2 \cup \cdots \cup \Phi_r .
\end{equation} 
Furthermore, there is a natural bijection 
$\rho\colon \MC{M}_\xi^{\om{sing}}(Q,\delta) \cong \{\Phi_1,\Phi_2,\cdots,\Phi_r\}$ and the local structure around the singularities can be described
as follows. Let $x\in \MC{M}_\xi^{\om{sing}}(Q,\delta)$ and let 
$\Gamma_x\subset \om{SU}(2)$ be the finite group associated with $\rho(x)$ under the McKay correspondence. Then there is an open neighborhood
$x\in U_x\subset \MC{M}_\xi(Q,\delta)$ and a homeomorphism $\phi_x\colon U_x\ra B_r(0)/\Gamma_x$, where $B_r(0)\subset \C^2$ is the open ball of radius $r$,
that restricts to a diffeomorphim
\[ \MC{M}_{\xi}^{\om{reg}}(Q,\delta)\supset (U_x-\{x\})\cong (B_r(0)-\{0\})/\Gamma_x  .\] 
\end{theorem}

The fact that $\MC{M}_\xi(Q,\delta)$ is non-singular when $\xi$ avoids
all the root walls $D_\theta$ for $\theta\in \Phi$ is the content
of \cite[Corollary~2.10]{Kronheimer89}. 

We give a brief outline of the proof of Theorem \ref{Singularity-Theorem} and, in particular, explain why we make the restriction $\xi_1=0$. The action of
the compact group $G_v$ on $\om{Rep}(\ovl{Q},v)$ extends to a linear action of the complexification $G_\delta^c = \prod_{i=0}^n \om{GL}(\delta_i,\C)$. Moreover, the hyper-K\"{a}hler moment map splits 
\[ \mu = (\mu_\R,\mu_\C)\colon \om{Rep}(\ovl{Q},\delta)\ra \R^3\otimes \fg_\delta \cong \fg_\delta \oplus \fg_\delta^c, \] 
where $\fg_\delta^c = \om{Lie}(G_\delta^c)$, and the second component
is a moment for the action of $G_\delta^c$ with respect to a complex
symplectic form on $\om{Rep}(\ovl{Q},\delta)$. In the situation where
the parameter $\xi = (\xi_1,\xi_2,\xi_3)\in \R^3\otimes \R^n$ satisfies
$\xi_1=0$, there is a homeomorphism between the hyper-K\"{a}hler quotient
$\MC{M}_\xi(Q,\delta)$ and the GIT quotient $\mu_\C^{-1}((\xi_2,\xi_3))//G_\delta^c$. The elements of the latter quotient have a representation 
theoretic interpretation. Indeed, if we write $\lambda = \xi_2+i\xi_3\in \C^I$, the points of $\mu_\C^{-1}(\lambda)//G_\delta^c$ are in natural bijection with
the isomorphism classes of semi-simple modules of dimension $\delta$ over the deformed preprojective algebra $\Pi^\lambda=\Pi^\lambda(Q)$ introduced in \cite{Crawley98}. Under these bijections the singularities in 
$\MC{M}_\xi(Q,\delta)$ correspond precisely to the non-simple,
semi-simple modules. Using a result of Crawley-Boevey
\cite{Crawley01} on the existence and uniqueness of simple $\Pi^\lambda$-modules, we are able to set up a bijection between the latter set and
the root systems in the statement of the theorem.

To establish the homeomorphisms $\phi_x\colon U_x\ra B_r(0)/\Gamma_x$ we employ
a result of \cite{Mayrand18} that essentially reduces the statement
to the determination of the complex symplectic slice (see Definition \ref{Def-Symp-Slice}) at a point $\wt{x}\in \mu^{-1}(0,\lambda)$ above $x$.
We should note that a result along these lines is given in \cite[Lemma~3.3]{Kronheimer89}, however, the proof given there seems to contain a gap that
we have been unable to close. For this reason we have chosen to rely on
the above mentioned result instead.     

The finite subgroups $\Gamma\subset \om{SU}(2)$ are called the binary
polyhedral groups. By restricting the canonical action of $\Gamma$
to the three-sphere $S^3\subset \C^2$ we obtain the binary
polyhedral spaces $S^3/\Gamma$. In Proposition \ref{Configuration-Prop}
we determine what kind of root space decomposition
\[ \Phi\cap \xi = \Phi_1\cup \cdots \cup \Phi_r \]
one can obtain by varying the parameter $\xi$. Combining this with
the above theorem we obtain the following constructive procedure for
hyper-K\"{a}hler manifolds with a number of ends modeled on
$(0,\infty)\times S^3/\Gamma$ for finite subgroups $\Gamma\subset \om{SU}(2)$. 
In the following statement we say that a subgraph $H$ of $G$ is a full 
subgraph if every edge in $G$ connecting a pair of vertices in $H$ belongs to $H$. 

\begin{theorem} \label{Bordism-Theorem}
Let $\Gamma_0,\Gamma_1,\cdots,\Gamma_r\subset \om{SU}(2)$
be finite subgroups and let $K_0,K_1,\cdots,K_r$ denote
the corresponding (non-extended) Dynkin graphs. Let $Q$ be an extended
Dynkin quiver with vertex set $I$, whose underlying unoriented graph is 
the extended version of $K_0$. Then if $K_1\sqcup K_2 \sqcup \cdots \sqcup K_r$ can be realized as a full subgraph of $K_0$, there exists a parameter
$\xi \in \R^3\otimes \R^I$ such that $X=\MC{M}_{\xi}^{\om{reg}}(Q,\delta)$
satisfies the following properties.
\begin{enumerate}[label=(\arabic*), ref=(\arabic*)]
\item $X$ is a connected hyper-K\"{a}hler manifold of dimension $4$.
\item There are disjoint open subsets $U_0,U_1,\cdots,U_r\subset X$ and
for each $0\leq i\leq r$ a diffeomorphism
\[ \phi_i\colon U_i\ra (0,\infty)\times S^3/\Gamma_i  .\]
\item The complement $Y = X-\bigcup_{i=0}^r U_i$ is a compact $4$-manifold with boundary components $S^3/\Gamma_i$ for $0\leq i\leq r$. 
\end{enumerate} 
\end{theorem}

Note that the diffeomorphism $\phi_i$, $0\leq i\leq r$, will generally not preserve the hyper-K\"{a}hler structure. 

We wish to briefly describe the gauge theoretic motivation for pursuing the above result. In \cite{Helle22} the author calculates the equivariant instanton Floer homology in the sense of \cite{Miller19} for the binary polyhedral spaces. The key geometric input needed for the calculations is a close understanding of appropriate moduli spaces of $\om{SU}(2)$-instantons over the cylinders $\R\times S^3/\Gamma$ for finite $\Gamma\subset \om{SU}(2)$. In \cite{Austin95} Austin tackled this problem using an equivariant version of the classical ADHM correspondence
(see \cite[Section~3.3]{DK90} or \cite{Atiyah79}). This work inspired
the generalized ADHM correspondence of Kronheimer and Nakajima \cite{KronheimerNakajima} that describes instanton moduli spaces associated with unitary bundles over the ALE-spaces as Nakajima quiver varieties (for this later reformulation see \cite{Nakajima94}).
To elaborate, if $Q$ is an extended Dynkin quiver whose underlying graph corresponds to $\Gamma$ under the McKay correspondence then by \cite[Corollary~3.2]{Kronheimer89}
one has
\[ \R\times S^3/\Gamma \cong (\C^2-\{0\})/\Gamma \cong \MC{M}^{\om{reg}}_0(Q,\delta), \]
where $\delta$ is the minimal imaginary root as before. With this in mind, the equivariant ADHM correspondence of \cite{Austin95} can be regarded as a degenerate case of the generalized ADHM correspondence of \cite{KronheimerNakajima}. These
two cases suggest that it should be possible to extend the ADHM correspondence to the singular situation considered in this paper as well. This conjectural leap would open up the possibility of studying
cobordism maps in equivariant Floer homology associated with the many explicit cobordisms obtained from the above theorem.

The paper is organized as follows. In section (2) we give the basic
definitions concerning hyper-K\"{a}hler manifolds and hyper-K\"{a}hler
reduction. In section (3) we introduce quivers and
quiver varieties and state the key results that will be needed concerning these. In section (4) we recall the basic elements of the complex representation theory of quivers. Afterwards, we give the definition of the deformed preprojective algebras $\Pi^\lambda(Q)$ and spell out the relation between the quiver variety $\MC{M}_{(0,\lambda)}(Q,v)$ and the isomorphism classes of semi-simple $\Pi^\lambda(Q)$-modules. Finally, we recall the key result of \cite{Crawley01} that eventually allow us to classify the singularities in $\MC{M}_{(0,\lambda)}(Q,v)$. In section (5) we give the construction of the extended 
Dynkin diagrams from the underlying Dynkin diagram and review the necessary root space theory of the associated root systems. 

Our original work starts in section $(6)$ where we establish the
bijection between the singularities in the (extended Dynkin) quiver varieties and the components in the corresponding root space decomposition as in
\eqref{Intro-RootSpaceDecomp}. In section (7) we establish the local models around the singularities using a result of \cite{Mayrand18} and give the
proof of Theorem \ref{Singularity-Theorem}. In the final section we
determine the possible configurations of singularities in the various
quiver varieties and complete the proof of Theorem \ref{Bordism-Theorem}. 
   
\section{Hyper-K\"{a}hler Reduction}
A hyper-K\"{a}hler manifold is a tuple $(M,g,I,J,K)$ consisting of
a smooth manifold $M$, a Riemannian metric $g$ and three
almost complex structure maps $I,J,K\colon TM\ra TM$ subject to the
following conditions:
\begin{enumerate}[label = (\alph*), ref= (\alph*)]
\item $I$, $J$ and $K$ are orthogonal with respect to $g$,
\item $IJK = -1_{TM}$ and
\item $\nabla^g I=\nabla^g J=\nabla^g K=0$, where $\nabla^g$ is the Levi-Civita
connection.
\end{enumerate} 
In particular, for each $S\in \{I,J,K\}$ the triple $(M,g,S)$ is a K\"{a}hler manifold with K\"{a}hler form $\omega_S$ given by $(\omega_S)_p(v,w)=g_p(Sv,w)$ for each $p\in M$ and $v,w\in T_pM$.

Following the terminology of \cite{Mayrand18} 
a tri-Hamiltonian hyper-K\"{a}hler 
manifold is a triple $(M,K,\mu)$ consisting of a hyper-K\"{a}hler manifold
$M$, a compact Lie group $K$ acting on $M$ preserving the hyper-K\"{a}hler
structure and a hyper-K\"{a}hler moment map 
$\mu=(\mu_I,\mu_J,\mu_K)\colon M\ra \R^3\otimes \mathfrak{k}^*$, where $\mathfrak{k}$ is the Lie algebra of $K$. Note that by definition $\mu$ is a hyper-K\"{a}hler moment map if and only if the components $\mu_I,\mu_J,\mu_K$ are moment maps for the corresponding symplectic forms $\omega_I,\omega_J,\omega_K$, respectively, in the sense familiar from symplectic geometry
(see for instance \cite{daSilva01}). 

The group $K$ acts on $\mathfrak{k}^*$ through the coadjoint action and
we denote the set of fixed points by $(\mathfrak{k}^*)^K$. 
For each $\xi\in \R^3\otimes (\mathfrak{k}^*)^K$ the fiber $\mu^{-1}(\xi)$
is $K$-invariant and the quotient space $\mu^{-1}(\xi)/K$ is called
a hyper-K\"{a}hler quotient.  

\begin{theorem}\cite{HKL87} \label{HK-Free-Theorem}
Let $(M,K,\mu)$ be a tri-Hamiltonian hyper-K\"{a}hler manifold and let
$\xi\in \R^3\otimes (\mathfrak{k}^*)^K$. If $K$ acts
freely on $\mu^{-1}(\xi)$, then the following holds true. 
\begin{enumerate}[label=(\alph*), ref=(\alph*)]
\item $\xi$ is a regular value for $\mu$ so that $\mu^{-1}(\xi)$ is a smooth
submanifold of $M$.
\item The quotient $\mu^{-1}(\xi)/K$ is a smooth manifold of dimension
$\om{dim} M -4\om{dim}K$ and the projection $\pi\colon \mu^{-1}(\xi)\ra \mu^{-1}(\xi)/K$ is a principal $K$-bundle. 
\item There is a unique hyper-K\"{a}hler structure on $\mu^{-1}(\xi)/K$
with K\"{a}hler forms $\omega_I',\omega_J',\omega_K'$ such that
$\pi^*(\omega_S') = \omega_S|_{\mu^{-1}(\xi)}$ for each $S\in \{I,J,K\}$.   
\end{enumerate} 
\end{theorem}

The passage from $(M,K,\mu)$ to $\mu^{-1}(\xi)/K$ for $\xi\in \R^3\otimes (\mathfrak{k}^*)^K$ is called hyper-K\"{a}hler reduction.
Even if the action of $K$ on $\mu^{-1}(\xi)$ fails to be free, the
hyper-K\"{a}hler quotient $X\coloneqq \mu^{-1}(\xi)/K$ admits a decomposition
into smooth hyper-K\"{a}hler manifolds of various dimensions (see \cite[Theorem~1.1]{Mayrand18}).
For our purpose it will be sufficient to note that if $U\subset M$ denotes
the open (possibly empty) set consisting of the free $K$-orbits, then
$\mu|_U\colon U\ra \R^3\otimes \mathfrak{k}^*$ is a moment map for the action
of $K$ on $U$, and therefore $(\mu^{-1}(\xi)\cap U)/K \eqqcolon X^{\om{reg}} \subset X$ carries the structure of a smooth hyper-K\"{a}hler manifold by the above theorem. The open subset $X^{\om{reg}}$ is called the regular set and
its closed complement $X^{\om{sing}} \coloneqq X-X^{\om{reg}}$ is called the singular set. 

We will only be interested in a very simple instance of the above procedure.
Let $V$ be a quaternionic vector space equipped with a compatible real
inner product $g\colon V\times V\ra \R$, that is, $V$ is a real vector space
equipped with three orthogonal endomorphisms $I,J,K\colon V\ra V$ satisfying 
the relations of the quaternion algebra:
\[ I^2=J^2=K^2 = IJK= -1_V .\]
Using the standard identification $T_pV\cong V$ for each $p\in V$, we
may regard $(V,g,I,J,K)$ as a flat hyper-K\"{a}hler manifold. Let
$K$ be a compact Lie group acting linearly on $V$ preserving $(g,I,J,K)$.
In this situation the unique hyper-K\"{a}hler moment map vanishing at $0\in V$,
$\mu = (\mu_I,\mu_J,\mu_K)\colon V\ra \R^3\otimes \mathfrak{k}^*$, is given
by 
\[ \mu_I(x)(\xi) = \frac12 \omega_I(\xi\cdot x,x)=\frac12 g(\xi\cdot Ix,x) \]
for $x\in V$, $\xi\in\mathfrak{k}$ and similarly for $\mu_J$ and $\mu_K$. We call the triple $(V,K,\mu)$ a linear tri-Hamiltonian hyper-K\"{a}hler manifold.

\section{Quiver Varieties}
A quiver is a finite directed graph $(Q,I,s,t)$, where $I$ is the set of
vertices, $Q$ is the set of edges and $s,t\colon Q\ra I$ are the source and
target maps. Given an edge $h\in Q$ with $s(h)=i\in I$ and $t(h)=j\in I$
we write $h\colon i\to j$. We will abuse notation slightly and refer to the
quiver simply as $Q$ or $(Q,I)$ letting $s$ and $t$ be implicit. The purpose
of this section is to fix our notation, define the quiver varieties of
interest and state a few
results needed for our later work. We will later restrict our attention to
the quivers specified in the following definition.

\begin{definition} An extended Dynkin quiver 
is a quiver $Q$ whose underlying unoriented graph is an extended 
Dynkin diagram of type $\wt{A}_n$, $\wt{D}_n$, $\wt{E}_6$, $\wt{E}_7$ or
$\wt{E}_8$. Similarly, a Dynkin quiver is a quiver whose underlying
unoriented graph is a Dynkin diagram of type $A_n$, $D_n$, $E_6$, $E_7$
or $E_8$. \end{definition}  

Let $(Q,I)$ be a quiver. For each $v=(v_i)_{i\in I}\in \Z^I_{\geq 0}$, called a dimension vector, define
\begin{align*}  \label{Def-RepSpace-Groups}
 \om{Rep}(Q,v) \coloneqq & \bigoplus_{h\in Q}\om{Hom}(\C^{v_{s(h)}},\C^{v_{t(h)}}) \\
G_v\coloneqq & \prod_{i\in I} U(v_i)  \\
G_v^c \coloneqq & \prod_{i\in I} \om{GL}(v_i,\C),
\end{align*}
where $U(v_i)\subset \om{GL}(v_i,\C)$ denotes the group of unitary matrices
for each $i\in I$. There is an evident inclusion $G_v\subset G_v^c$ witnessing
the fact that $G_v^c$ is the complexification of $G_v$. The Lie algebras
$\fg_v\coloneqq \om{Lie}(G_v)$ and $\fg_v^c \coloneqq  \om{Lie}(G_v^c)$ are given by
\[ \fg_v = \bigoplus_{i\in I} \mathfrak{u}(v_i) \;\; \mbox{ and } \;\; 
\fg_v^c = \bigoplus_{i\in I} \om{End}(\C^{v_i}) .\]
The group $G_v^c$ acts linearly on $\om{Rep}(Q,v)$ by the formula
\[ g\cdot x = (g_{t(h)}x_h g_{s(h)}^{-1})_{h\in Q} \; \mbox{ for }
g=(g_i)_{i\in I}\in G_v^c \; \mbox{ and } \;x = (x_h)_{h\in Q}\in \om{Rep}(Q,v). \]
The subgroup $G_v$ acts by restriction along the inclusion $G_v\subset G_v^c$. The space $\om{Rep}(Q,v)$ carries a Hermitian inner product preserved by the action of $G_v$. Explicitly,
\[ (x,y) = \sum_{h\in Q} \om{tr}(x_hy_h^*) , \]
where $\om{tr}$ is the trace and $y_h^*$ is the adjoint of $y_h$ with respect to the standard Hermitian inner product on $\C^{v_i}$ for $i\in I$. 

\begin{definition} \label{Def-Opposite-Doubled-Quiver}
Let $Q$ be a quiver. The opposite quiver $Q^{op}$ is defined by taking
the same set of vertices and reverse the orientation of each edge.
For an edge $h\in Q$ the opposite edge is denoted by $\ovl{h}\in Q^{op}$.
The doubled quiver $\ovl{Q}$ is defined by taking the same set of
vertices and let the set of edges be $Q\cup Q^{op}$. The orientation map 
$\eps\colon \ovl{Q}\ra \{\pm1\}$ is defined by
$\eps(h)=+1$ if $h\in Q$ and $\eps(h)=-1$ if $h\in Q^{op}$.
\end{definition} 

We extend the bijection $Q\ra Q^{op}$, $h\mapsto \ovl{h}$, to an involution of 
$\ovl{Q}$ by setting $\ovl{h_2}=h_1$ if and only if $\ovl{h_1}=h_2$ for $h_1\in Q$ and $h_2\in Q^{op}$.

Given a quiver $Q$ with vertex set $I$ and a dimension vector $v\in \Z^I_{\geq 0}$, there is a natural decomposition
\[ \om{Rep}(\ovl{Q},v) = \om{Rep}(Q,v)\oplus \om{Rep}(Q^{op},v) . \]
This gives rise to a quaternionic structure $J\colon \om{Rep}(\ovl{Q},v)\ra \om{Rep}(\ovl{Q},v)$. In terms of the above decomposition, $J$ is given by
$J(x,y) = (-y^*,x^*)$, where $(x^*)_h \coloneqq (x_{\ovl{h}})^*$ and similarly
for $y$. The action of $G_v$ commutes with this quaternionic structure and we may therefore regard $\om{Rep}(\ovl{Q},v)$ as a
quaternionic representation of the compact group $G_v$.   
The components of the unique hyper-K\"{a}hler moment map 
$\mu=(\mu_\R,\mu_\C)\coloneqq \om{Rep}(\ovl{Q},v)\ra \fg_v\oplus \fg_v^c$ vanishing at zero, where the Lie algebras are identified with their duals using the trace pairing, have the explicit forms
\cite[p.~370]{Nakajima94}
\begin{align} \label{MomentMap-Eq}
\mu_\R(x) &= \frac{\sqrt{-1}}{2}\left( \sum_{h\in t^{-1}(i)} x_hx_h^*-x_{\ovl{h}}^*x_{\ovl{h}} \right)_{i\in I}  \nonumber  \\ 
\mu_\C(x) &= \left( \sum_{h\in t^{-1}(i)} \eps(h)x_hx_{\ovl{h}} \right)_{i\in I}.  
\end{align} 
In the terminology of the previous section $(\om{Rep}(\ovl{Q},v),G_v,\mu)$
is a linear tri-Hamiltonian hyper-K\"{a}hler manifold.

Under the identifications of $\fg_v$ and $\fg_v^c$ with their dual spaces,
the subspaces fixed under the coadjoint action are identified with
the centers $Z(\fg_v)$ and $Z(\fg_v^c)$. There are natural maps
$\R^I\ra Z(\fg_v)$ and $\C^I\ra Z(\fg_v^c)$ given by
\begin{align*}
(\xi_i)_{i\in I}\in \R^I \mapsto & (\sqrt{-1}\xi_i \om{Id}_{\C^{v_i}})_{i\in I} \in \bigoplus_{i\in I} Z(\mathfrak{u}(v_i)) \\
(\lambda_i)_{i\in I}\in \C^I\mapsto & (\lambda_i \om{Id}_{\C^{v_i}})_{i\in I}
\in \bigoplus_{i\in I} Z(\om{End}(\C^{v_i}).
\end{align*}
If $v_i\neq 0$ for each $i\in I$, then both of these are isomorphisms. Otherwise, they restrict to isomorphisms from $\R^{\om{supp}v}$ and $\C^{\om{supp}v}$, respectively, where $\om{supp}v = \{i\in I: v_i\neq 0\}$.
For any dimension vector $v\in \Z^I_{\geq 0}$ we will tacitly regard elements
$\xi\in \R^I$ and $\lambda\in \C^I$ as elements
of $Z(\fg_v)$ and $Z(\fg^c_v)$, respectively, using the above maps.

\begin{definition} Let $Q$ be a quiver with vertex set $I$. For any
dimension vector $v\in \Z^I_{\geq 0}$ and parameter
$\xi = (\xi_\R,\xi_\C)\in \R^I\oplus \C^I$ define
\[ \MC{M}_\xi(Q,v) \coloneqq  \mu^{-1}(\xi)/G_v  .\]
These hyper-K\"{a}hler quotients are called (unframed) quiver varieties.
\end{definition}
 
\begin{remark} In \cite{Nakajima94} Nakajima defines what one may call
framed quiver varieties $\MC{M}_\xi(v,w)$ associated with a quiver $Q$ with vertex set $I$ and two dimension vectors $v,w\in \Z^I$. The above defined spaces $\MC{M}_\xi(Q,v)$ correspond to his $\MC{M}_\xi(v,0)$. According to 
\cite[p.~261]{Crawley01} the spaces $\MC{M}_\xi(v,w)$ can be
expressed as $\MC{M}_{\xi'}(\ovl{Q}_1,v')$, where $Q_1$ is a quiver
obtained from $Q$ by adjoining a single vertex and a number of arrows depending
on $w$. There is therefore no loss in generality in only considering these
(unframed) quivers. 
\end{remark} 

The subgroup $T$ of scalars, i.e., $U(1)\cong T \subset G_v$, acts trivially
on $\om{Rep}(\ovl{Q},v)$ so the action factors through $G_v\ra G_v/T \eqqcolon G_v'$. As explained in the previous section we obtain a decomposition
\[ \MC{M}_\xi(Q,v) = \MC{M}_\xi^{\om{reg}}(Q,v) \cup \MC{M}_\xi^{\om{sing}}(Q,v), \]
where the regular set $\MC{M}_\xi^{\om{reg}}(Q,v)$ is the image
of the free $G_v'$-orbits in $\mu^{-1}(\xi)$ or equivalently
the points $x\in \mu^{-1}(\xi)$ with stabilizer $T$ in $G_v$. The
regular set is open in $\MC{M}_\xi(Q,v)$ and carries the structure
of a smooth hyper-K\"{a}hler manifold. The singular set 
$\MC{M}_\xi^{\om{sing}}(Q,v)$ is the closed complement of the
regular set. 

The fact that the action of $G_v$ factors through $G_v'$ has another important
implication, namely, that the moment map 
$\mu\colon \om{Rep}(\ovl{Q},v)\ra \R^3\otimes \fg_v$ takes
values in the subspace $\fg_{v,0}\subset \fg_v$ corresponding
to $(\fg_v')^* =\om{Lie}(G_v')^*$ under the isomorphism $\fg_v^*\cong \fg_v$.
This subspace consists precisely of the $(a_i)_{i\in I}\in \fg_v$ satisfying 
$\sum_{i\in I} \om{tr}(a_i)=0$.
A parameter $\xi=(\xi_1,\xi_2,\xi_3)\in \R^3\otimes \R^I$ corresponds
to an element satisfying this condition precisely when
\[ v\cdot \xi_k = \sum_{i\in I} \om{tr}((\xi_k)_i \om{Id}_{\C^{v_i}}) =0 \; \mbox{ for } \; k=1,2,3, \]
where $\cdot$ denotes the usual scalar product. For each $\theta\in \Z^I$
define
\[ D_\theta = \{u \in \R^I:u\cdot \theta=0\}\subset \R^I.  \]
The above then amounts to the fact that $\mu^{-1}(\xi)=\emptyset$ whenever
$\xi_k \notin D_v$ for some $1\leq k\leq 3$. However, for most
parameters $\xi\in \R^3\otimes D_v$ the space $\MC{M}_\xi(Q,v)$ will
be a smooth hyper-K\"{a}hler manifold. To state the relevant result
we have to recall the definition of the symmetric bilinear form
associated with a quiver (see for instance \cite[Section~2]{Crawley01}). 

\begin{definition} \label{Def-SB-form}
Let $Q$ be a quiver with vertex set $I$. The symmetric
bilinear form $(\cdot,\cdot)\colon \Z^I\times  \Z^I\ra \Z$ associated with the quiver is defined by
\[ (v,w) \coloneqq 2\sum_{i\in I}v_i w_i - \sum_{h\in \ovl{Q}} v_{s(h)}w_{t(h)}
\;\; \mbox{ for } \;\; v,w\in \Z^I .\]
\end{definition}

If we identify the set of vertices $I\cong \{1,2,\cdots,n\}$ for some $n\in \N$ and let $A = (a_{ij})$ be the adjacency matrix of the unoriented graph underlying $Q$, i.e., $a_{ij}=a_{ji}$ is the number of edges connecting 
$i$ and $j$, then $(v,w) = 2v\cdot w -v\cdot Aw$. Alternatively, $(v,w) = v\cdot Cw$ where $C=2\om{id}-A$. The symmetric bilinear form therefore only depends on the underlying unoriented graph. If $Q$ is a (extended) Dynkin quiver, then $C$ is the Cartan matrix associated with the corresponding (extended) Dynkin diagram.

The following theorem is \cite[Theorem~2.8]{Nakajima94} adapted to the
unframed setting. Let $\Z^I$ be partially ordered
by $v\leq w$ if and only if $v_i\leq w_i$ for each $i\in I$. 

\begin{theorem} \label{Smooth-Quotient-Theorem}
Let $Q$ be a quiver with vertex set $I$. Given a dimension
vector $v\in \Z^I_{\geq 0}$ define
\[ R_+(v) = \{\theta\in \Z^n : 0<\theta <v \; \mbox{ and } \; (\theta,\theta)\leq 2\} . \]
Then if
\[ \xi \in \R^3\otimes D_v - \left( \bigcup_{\theta\in R_+(v)} \R^3\otimes (D_v\cap D_\theta)  \right), \]
the group $G_v'$ acts freely on $\mu^{-1}(\xi)\subset \om{Rep}(\ovl{Q},v)$ and
the quiver variety $\MC{M}_\xi(Q,v)$ is a (possibly empty) 
smooth hyper-K\"{a}hler manifold of dimension $4-2(v,v)$.
\end{theorem}  

Let $(Q,I)$ be a quiver and fix a dimension vector $v\in \Z^I_{\geq 0}$.
The complex Lie group $G_v^c$ acts on $\om{Rep}(\ovl{Q},v)$ preserving the complex symplectic form $\omega_\C$ given by the formula
\begin{equation} \label{Symp-Complex-Formula} 
\omega_\C(x,y) = \sum_{h\in \ovl{Q}} \eps(h)\om{tr}(x_hy_{\ovl{h}})
\;\; \mbox{ for } \;\; x,y\in \om{Rep}(\ovl{Q},v).
\end{equation}
The corresponding moment map is precisely the component $\mu_\C\colon \om{Rep}(\ovl{Q},v)\ra \fg_v^c$ in \eqref{MomentMap-Eq}. From the given formula
it is clear that $\mu_\C$ is algebraic and therefore $\mu_\C^{-1}(\xi_\C)$
carries the structure of an affine variety for each $\xi_\C\in \C^I$. The action of the reductive group $G_v^c$ is algebraic so there is a complex analytic quotient $\mu_\C^{-1}(\xi_\C)\ra \mu^{-1} (\xi_\C)//G_v^c$. This is the analytification of the affine GIT quotient
\[ \om{Spec}\C[\mu_\C^{-1}(\xi_\C)] \ra \om{Spec}(\C[\mu_\C^{-1}(\xi_\C)]^{G_v^c}). \]
We will need a few standard facts concerning this construction
(see for instance \cite[Chapter~6]{Dolgachev03} for the algebraic
side of the story and \cite[Section~2.4.1]{Mayrand18} and the references
contained therein for the analytical perspective).  

\begin{lemma} \label{GIT-quotient-Lemma}
As a topological space $\mu^{-1}_\C(\xi_\C)//G_v^c$ is homeomorphic to the quotient space $\mu^{-1}_\C(\xi_\C)/\sim$ where 
$x\sim y$ if and only if $\ovl{G_v^c\cdot x}\cap \ovl{G_v^c\cdot y}\neq \emptyset$. Let $q\colon \mu^{-1}_\C(\xi_\C)\ra \mu^{-1}_\C(\xi_\C)//G_v^c$ denote the quotient map. Then each fiber $q^{-1}(x)$ contains a unique closed
orbit $G_v^c\cdot \tilde{x}$, and if $y\in q^{-1}(x)$ then $G_v^c\cdot x\subset \ovl{G^c_v\cdot y}$. \end{lemma}  

In this setting we have the following result comparing the analytic quotient
and the hyper-K\"{a}hler quotient. 

\begin{theorem}\cite[Theorem~3.1]{Nakajima94} \label{Quotient-Equivalence-Theorem}
Let $Q$ be a quiver with vertex set $I$ and let
$v\in \Z^I$ be a dimension vector. Then for each $\xi_\C \in \C^I$ the inclusion $\mu^{-1}(0,\xi_\C)=\mu_\R^{-1}(0)\cap \mu_\C^{-1}(\xi_\C)\inj \mu_\C^{-1}(\xi_\C)$ descends to a homeomorphism
\[ \MC{M}_{(0,\xi_\C)}(Q,v) = \left( \mu_\R^{-1}(0)\cap \mu_\C^{-1}(\xi_\C) \right)/G_v \cong \mu_\C^{-1}(\xi_\C)//G_v^c  .\]
Moreover, each closed orbit $G_v^c\cdot x\subset \mu_\C^{-1}(\xi_\C)$
intersects $\mu^{-1}_\R(0)$ in a unique $G_v$-orbit. 
\end{theorem}

\begin{remark} The final statement is not explicitly stated in
\cite{Nakajima94}, but seems to be well-known. See for instance
\cite[Proposition~2.2]{Sjamaar95}.
\end{remark}

The above result implies that $\MC{M}_{(0,\xi_\C)}(Q,v)$
carries the structure of a complex analytic space. We will have use for
one final result. Let $v\in \Z^I_{\geq0}$ be a fixed dimension vector
and let $\xi_\C\in \C^I$ such that $\om{Re}\xi_\C,\om{Im}\xi_\C\in D_v$.
Choose $\xi_\R \in D_v-\bigcup_{\theta\in R_+(v)} D_\theta$ and set
$\xi = (0,\xi_\C)$ and $\wt{\xi}=(\xi_\R,\xi_\C)$. The space
$\MC{M}_{\wt{\xi}}(Q,v)$ is a smooth hyper-K\"{a}hler manifold by Theorem \ref{Smooth-Quotient-Theorem}.
The inclusion 
\[ \mu^{-1}(\wt{\xi})=\mu_\R^{-1}(\xi_\R)\cap \mu_\C^{-1}(\xi_\C) \inj
\mu_\C^{-1}(\xi_\C)  \]
induces a map $\pi\colon \MC{M}_{\wt{\xi}}(Q,v)\ra \mu_\C^{-1}(\xi_\C)//G_v^c \cong \MC{M}_\xi(Q,v)$. In the following result we regard
$\MC{M}_{\wt{\xi}}(Q,v)$ as a complex manifold by fixing the complex
structure induced by the standard complex vector space structure of 
$\om{Rep}(\ovl{Q},v)$. 

\begin{theorem} \cite[Theorem~4.1]{Nakajima94} \label{Resolution-Sing-Theorem}
The map $\pi$ is holomorphic and provided $\MC{M}_\xi^{\om{reg}}(Q,v)$ is 
nonempty, it is a resolution of singularities, that is,
\begin{enumerate}[label=(\arabic*), ref=(\arabic*)]
\item $\pi\colon \MC{M}_{\wt{\xi}}(Q,v)\ra \MC{M}_\xi(Q,v)$ is proper,
\item $\pi$ induces an isomorphism $\pi^{-1}(\MC{M}_{\xi}^{\om{reg}}(Q,v))\cong \MC{M}_\xi^{\om{reg}}(Q,v)$ and
\item $\pi^{-1}(\MC{M}_\xi^{\om{reg}}(Q,v))$ is a dense subset of $\MC{M}_{\wt{\xi}}(Q,v)$. 
\end{enumerate}
\end{theorem} 

\section{Representations of Quivers}
We briefly recall a few basic notions concerning the representation theory
of quivers. An excellent reference for this material is \cite{Brion12}.  Afterwards we give the definition of the deformed preprojective algebras $\Pi^\lambda = \Pi^\lambda(Q)$ of \cite{Crawley98} and spell out the correspondence between $\MC{M}_{(0,\lambda)}(Q,v)$ and the isomorphism
classes of semi-simple $\Pi^\lambda$-modules. Finally, we recall the
construction of the root system associated with a quiver and state the key result of \cite{Crawley01} relevant for our purpose.

A (complex) representation of a quiver $Q$ is a pair $(V,f)$ where
$V=(V_i)_{i\in I}$ is a family of complex vector spaces and
$f=(f_h\colon V_{s(h)}\ra V_{t(h)})_{h\in Q}$ is a family of linear maps.
We will only be concerned with finite dimensional representations, i.e.,
$V_i$ is finite dimensional for each $i\in I$. 
The dimension of a representation $(V,f)$ is $\om{dim}V\coloneqq (\om{dim}(V_i))_{i\in I}\in \Z^I_{\geq 0}$. A homomorphism $u\colon (V,f)\ra (W,g)$ of representations is a collection of linear maps $u_i\colon V_i\ra W_i$ for $i\in I$ such that $f_hu_{s(h)}=u_{t(h)}g_h$ for each $h\in Q$. We therefore have a category of complex representations of $Q$. This category is equivalent to the category of left modules over the quiver algebra $\C Q$: the complex algebra generated by $\{e_i:i\in I\}$ and $\{h:h\in Q\}$
subject to the relations
\[ e_ie_j = \delta_{ij}e_i, \;\; e_ih = \delta_{it(h)}h \;\; \mbox{ and } \;\;  
 he_j = \delta_{s(h)j}h  \]
for all $i,j\in I$ and $h\in Q$, where $\delta_{ij}=1$ if $i=j$ and $\delta_{ij}=0$ otherwise. The $\{e_i\}_{i\in I}$ is a complete set of 
mutually orthogonal idempotents, in particular $1_{\C Q}=\sum_{i\in I}e_i$.

We briefly recall the equivalence between representations of $Q$ and left
$\C Q$-modules. Let $(V,f)$ be a representation of $Q$ and put $X=\oplus_{i\in I}V_i$. For each $i\in I$ let $\iota_i\colon V_i\ra X$ and $\pi_i\colon X\ra V_i$ denote the inclusion and projection, respectively. Define
$\rho\colon \C Q\ra \om{End}_\C (X)$ by $\rho(e_i) = \iota_i\circ \pi_i$ for each $i\in I$ and $\rho(h) = \iota_{t(h)}\circ f_h\circ \pi_{s(h)}$ for each $h\in Q$. One may then verify that $\rho$ is a well-defined homomorphism of 
$\C$-algebras and therefore
endows $X$ with a $\C Q$-module structure. One may recover $(V,f)$
from $(X,\rho)$ by setting $V_i = e_iX$ for $i\in I$ and
$f_h=\pi_{t(h)}\circ \rho(h)\circ \iota_{s(h)}$ for $h\in Q$. 
With this in mind, we will pass freely between the notion of a $Q$ representation and a $\C Q$-module. 

A $\C Q$-module $X$ of dimension $v\in \Z^I$ defines a unique $G_v^c$-orbit
$\MC{O}_X\subset \om{Rep}(Q,v)$. A representative $x$ for the orbit is
obtained by choosing a basis for $V_i = e_iX$, thereby identifying
$V_i\cong \C^{v_i}$, for each $i\in I$ and then letting $x_h\colon \C^{v_{s(h)}}\ra \C^{v_{t(h)}}$ be the corresponding linear maps. 
The correspondence $X\mapsto \MC{O}_X$ sets up a bijection
between the isomorphism classes of $\C Q$-modules of dimension $v$
and the set of $G_v^c$-orbits in $\om{Rep}(Q,v)$. Given a parameter 
$\lambda\in \C^I$ the $G_v^c$-orbits in $\mu_\C^{-1}(\lambda)\subset \om{Rep}(\ovl{Q},v)$ have a representation theoretic interpretation as well.  

\begin{definition}\cite[p.~611]{Crawley98} 
Let $Q$ be a quiver with vertex set $I$. The deformed preprojective algebra $\Pi^\lambda=\Pi^\lambda(Q)$
of weight $\lambda\in \C^I$ is defined to be the quotient of the quiver algebra 
$\C \ovl{Q}$ by the two-sided ideal generated by
\[ c = \sum_{i\in I} \lambda_i e_i - \sum_{h\in Q} [h,\ovl{h}]  .\]
\end{definition}

Observe that there is a decomposition $c=\sum_i c_i$ where
\[ c_i = e_i\left(\lambda_i1_{\C \ovl{Q}}-\sum_{h\in t^{-1}(i)} \eps(h)h\ovl{h}\right). \] 
In view of the formula \eqref{MomentMap-Eq} for $\mu_\C$, it is not hard to
see that the $G_v^c$-orbit of a $\C\ovl{Q}$-module $X$ is contained
in $\mu^{-1}(\lambda)$ precisely when $X$ descends to a $\Pi^\lambda$-module
along the projection $\C \ovl{Q}\ra \Pi^\lambda$. Therefore, the
$G^c_v$-orbits in $\mu_\C^{-1}(\lambda)\subset \om{Rep}(\ovl{Q},v)$ are
in natural bijection with the isomorphism classes of $\Pi^\lambda$-modules of dimension $v$. 

We have the following result describing the closed $G^c_v$-orbits
in $\om{Rep}(Q,v)$ (see for instance \cite[Section~2]{Brion12} for a proof).
Note that a $G_v^c$-orbit is closed in the Zariski topology if and only
if it is closed in the analytic topology. 

\begin{proposition} \label{Closed-Orbit-Prop}
Let $Q$ be a quiver with vertex set $I$ and let
$X$ be a finite dimensional $\C Q$-module of dimension $v\in \Z^I_{\geq 0}$.
Let $\MC{O}_X$ denote the orbit corresponding to the isomorphism
class of $X$ in $\om{Rep}(Q,v)$.
Then $\MC{O}_X$ is closed if and only if $X$ is semi-simple.
Moreover, let
\[ 0=X_0\subset X_1\subset X_2\subset \cdots \subset X_n = X \]
be a composition series for $X$, i.e., each quotient $X_k/X_{k-1}$, $1\leq k\leq n$, is a simple module, and let
$X_{ss}=\bigoplus_{i=1}^n X_i/X_{i-1}$ be the semi-simplification of $X$.
Then $\MC{O}_{X_{ss}}$ is the unique closed orbit contained in the closure 
of $\MC{O}_X$. 
\end{proposition}

Let $\MC{SS}(\Pi^\lambda,v)$ denote the set of isomorphism classes of
semi-simple $\Pi^\lambda$-modules of dimension $v$. For a semi-simple
$\Pi^\lambda$-module $X$ we let $[X]$ denote its isomorphism class
in $\MC{SS}(\Pi^\lambda,v)$.  

\begin{proposition} \label{Semi-Simp-Quiver-Eq-Prop}
Let $Q$ be a quiver with vertex set $I$ and let
$\Pi^\lambda$ be the associated deformed preprojective algebra of
weight $\lambda\in \C^I$. Then for each dimension vector $v\in \Z^I$,
the map
\[ \rho\colon \MC{M}_{(0,\lambda)}(Q,v)\ra \MC{SS}(\Pi^\lambda,v), \]
that assigns to a point $x\in \MC{M}_{(0,\lambda)}(Q,v)$ the isomorphism class
of the $\Pi^\lambda$-module corresponding to any point
$\tilde{x}\in \mu^{-1}(0,\lambda)$ in the fiber over $x$, is a well-defined bijection. 

Moreover, if $\rho(x)=[X]$ and $X=\bigoplus_{j=1}^k n_jX_j$ with the
$X_j$ simple and $n_j\in \N$, then for any point $\tilde{x}\in \mu^{-1}(0,\lambda)$ above $x$ there are isomorphisms
\[ (G_v)_{\tilde{x}} \cong \prod_{j=1}^k U(n_j) \;\; \mbox{ and }
  (G_v^c)_{\tilde{x}} \cong \prod_{j=1}^k \om{GL}(n_j,\C) .\]
In particular, $x\in \MC{M}^{\om{reg}}_{(0,\lambda)}(Q,v)$ if and only if
$X$ is simple.  
\end{proposition} 
\begin{proof} We divide the proof into four steps. The first sentence
in each step is a claim that we then go on to verify.  

\textbf{Step 1}: The rule $[X]\mapsto \MC{O}_X\subset \mu_\C^{-1}(\lambda)$
defines a bijection between $\MC{SS}(\Pi^\lambda,v)$ and the set of
closed $G_v^c$-orbits in $\mu_\C^{-1}(\lambda)$. We have seen that
the given rule sets up a bijection between the set of isomorphism classes
of $\Pi^\lambda$-modules of dimension $v$ and the $G_v^c$-orbits contained
in $\mu_\C^{-1}(\lambda)$. Since a $\Pi^\lambda$-module
$X$ is semi-simple if and only if it is semi-simple as a $\C\ovl{Q}$-module,
Proposition \ref{Closed-Orbit-Prop} ensures that this bijection restricts to
a bijection between the isomorphism classes of the semi-simple $\Pi^\lambda$-modules and the closed $G_v^c$-orbits in $\mu_\C^{-1}(\lambda)$.  

\textbf{Step 2}: The rule $(G_v\cdot x)\mapsto (G_v^c\cdot x)$ for
$x\in \mu^{-1}(0,\lambda)$ defines a bijection between the $G_v$-orbits 
in $\mu^{-1}(0,\lambda)$ and the closed $G_v^c$-orbits in
$\mu_\C^{-1}(\lambda)$. For any dimension vector $v\in \Z^I$ we have a commutative diagram
\[ \begin{tikzcd} \mu_\R^{-1}(0)\cap \mu_\C^{-1}(\lambda) \arrow{r}{i}
\arrow{d}{p} & \mu_\C^{-1}(\lambda) \arrow{d}{q} \\
\MC{M}_{(0,\lambda)}(Q,v) \arrow{r}{j} &
\mu^{-1}_\C(\lambda)//G_v^c ,\end{tikzcd} \]
where $p$ and $q$ are the quotient maps, $i$ is the inclusion and $j$
is the induced map between the quotients. According to Theorem 
\ref{Quotient-Equivalence-Theorem} the map $j$ is a homeomorphism
and in particular a bijection. Therefore, the only thing we need to
prove is that for each $x\in \mu_\R^{-1}(0)\cap \mu_\C^{-1}(\lambda)$
the orbit $G_v^c\cdot x\subset \mu_\C^{-1}(\lambda)$ is closed.
By Lemma \ref{GIT-quotient-Lemma} there is a unique closed orbit
$G^c_v\cdot y \subset q^{-1}q(i(x))$. Moreover, by the second statement
in Theorem \ref{Quotient-Equivalence-Theorem} we may assume that
$y=i(z)$ for some $z\in \mu_\R^{-1}(0)\cap \mu_\C^{-1}(\lambda)$.
Then as $jp(x)=qi(x)=qi(z)=jp(z)$ and $j$ is injective we conclude
that $p(x)=p(z)$ and hence $G_v\cdot x=G_v\cdot z$. This implies that
$G_v^c\cdot x = G_v^c\cdot z$ and as the latter orbit is closed by
construction the claim has been verified.

\textbf{Step 3}: The map $\rho\colon \MC{M}_{(0,\lambda)}(Q,v)\ra \MC{SS}(\Pi^\lambda,v)$
is a well-defined bijection. Let $p\colon \mu^{-1}(0,\lambda)\ra \MC{M}_{(0,\lambda)}(Q,v)$ denote the quotient map as in the above diagram. The map sending
$x\in \MC{M}_{(0,\lambda)}(Q,v)$ to the $G_v$-orbit $p^{-1}(x)\subset \mu^{-1}(0,\lambda)$ is clearly a bijection. The map $\rho$ sending a
point $x\in \MC{M}_{(0,\lambda)}(Q,v)$ to the isomorphism class of the
$\Pi^\lambda$-module associated with any choice of $\tilde{x}\in p^{-1}(x)$
is then precisely the composition of the bijection $x\mapsto p^{-1}(x)=G_v\cdot \tilde{x}$, the bijection of step $2$ and the inverse of the bijection of step
$1$. It is then clear that $\rho$ is a well-defined bijection. 

\textbf{Step 4}: If $\rho(x) = [X]$ and $X=\sum_{j=1}^k n_jX_j$ is a decomposition of $X$ into simple modules, then for any $\tilde{x}\in p^{-1}(x)$
it holds true that
\[ (G_v)_{\tilde{x}}\cong \prod_{j=1}^k U(n_j) \;\; \mbox{ and }\;\; (G_v^c)_{\tilde{x}}\cong \prod_{j=1}^k \om{GL}(n_j,\C) .\]
Let $y\in \mu_\C^{-1}(\lambda)\subset \om{Rep}(\ovl{Q},v)$ and denote the
corresponding $\Pi^\lambda$-module by $Y$. It is then easy to see that
the stabilizer $(G_v^c)_y$ coincides with the module theoretic automorphism group $\om{Aut}_{\Pi^\lambda}(Y)$.
If $Y$ is semi-simple and $Y = \oplus_{j=1}^k n_jY_j$ is a decomposition
into simple modules, it follows by Schur's lemma that
\[ \om{Aut}_{\Pi^\lambda}(Y) \cong \prod_{j=1}^k \om{GL}(n_j,\C)  .\]
Let $x\in \MC{M}_{(0,\lambda)}$, let
$\tilde{x}\in \mu_\R^{-1}(0)\cap \mu_\C^{-1}(\lambda)$ be a point above
$x$ and let $X=\sum_{j=1}^kn_jX_k$ be the corresponding semi-simple
$\Pi^\lambda$-module decomposed into simple summands. From the above
considerations we may deduce that there is an isomorphism
$(G_v^c)_{\tilde{x}}\cong \prod_{j=1}^k \om{GL}(n_j,\C)$. 
For any point $y\in \mu_\R^{-1}(0)$ it holds true
that the inclusion of stabilizers $\iota\colon (G_v)_y\inj (G_v^c)_y$
induces an isomorphism between the complexification of $(G_v)_y$ and
$(G_v^c)_y$ (see \cite[Proposition~1.6]{Sjamaar95}). Applying
this in the situation above we deduce that
$\prod_{j=1}^k \om{GL}(n_j,\C)$ is isomorphic to the complexification
of $(G_v)_{\tilde{x}}$. In particular, $(G_v)_{\tilde{x}}$ is isomorphic
to a maximal compact subgroup of $\prod_{j=1}^k \om{Gl}_{n_j}(\C)$
and as all such subgroups are conjugate we deduce that there is
an isomorphism
\[ (G_v)_{\tilde{x}}\cong \prod_{j=1}^k U(n_j)  .\]
This completes the final step and hence the proof.     
\end{proof} 

In \cite{Crawley01} Crawley-Boevey gives a strong result on the
existence and uniqueness of simple $\Pi^\lambda$-modules. To state the
result we need to recall the construction of the root system associated
with a quiver. Here we follow \cite[Section~2]{Crawley01}.

Let $Q$ be a quiver with vertex set $I$ and let $(\cdot,\cdot)\colon \Z^I\times \Z^I\ra \Z$ be the associated symmetric bilinear form of Definition
\ref{Def-SB-form}. Let $\{\eps_i\in \Z^I:i\in I\}$ denote the standard
basis of $\Z^I$, that is, $(\eps_i)_j =\delta_{ij}$ for $i,j\in I$.
To simplify the exposition slightly we will assume that $Q$ contains no
edge loops, i.e., there is no $h\in Q$ with $s(h)=t(h)$. This is valid in the case of (extended) Dynkin quivers. Note that this condition implies that 
$(\eps_i,\eps_i)=2$ for each $i\in I$.

For each $i\in I$ there is a reflection $s_i\colon \Z^I\ra \Z^i$ defined by $s_i(v) = v-(v,\eps_i)\eps_i$. These reflections generate a finite subgroup $W\subset \om{Aut}_\Z(\Z^I)$ called the Weyl group.
The action of the Weyl group on $\Z^I$ preserves the symmetric
bilinear form associated with the quiver.  
The support of $\alpha\in \Z^I$ is the full subquiver of $Q$ with vertex set
$\{ i\in I:\alpha_i\neq 0\}$. The fundamental domain $F\subset \Z^I_{\geq 0}-\{0\}$ is then defined to be the set of $\alpha\in \Z^I_{\geq 0}$ with connected support satisfying $(\alpha,\eps_i)\leq 0$ for each $i\in I$. 
The root system associated with the quiver $Q$ is defined to be
$\Phi \coloneqq \Phi^{\om{re}}\cup \Phi^{\om{im}}\subset \Z^I$ where
\[ \Phi^{\om{re}} = \bigcup_{i\in I} W\cdot \eps_i \;\; \mbox{ and } \;\;
\Phi^{\om{im}} = W\cdot (F\cup -F) .\]
The elements of $\Phi^{\om{re}}$ are called real roots and the elements of
$\Phi^{\om{im}}$ are called imaginary roots. One may show that there is
a decomposition $\Phi=\Phi^+\cup \Phi^-$ into positive and negative roots,
where a root $\alpha$ is positive (respectively negative) if $\alpha\in \Z^I_{\geq 0}$ (respectively $\alpha\in \Z^I_{\leq 0}$). We record the
following elementary fact.  

\begin{lemma} \label{RootLemma0}
For each $\alpha\in \Phi^{\om{re}}$ it holds true that $(\alpha,\alpha)=2$. For each $\beta\in \Phi^{\om{im}}$ it holds true that $(\beta,\beta)\leq 0$.
\end{lemma} 
\begin{proof} As already noted $(\eps_i,\eps_i)=2$ for each $i\in I$. The
first assertion now follows from the fact that each $\alpha\in \Phi^{\om{re}}$
may be expressed in the form $w\cdot \eps_i$ for some $w\in W$ and $i\in I$.
For the second assertion we may assume without loss of generality that
$\beta\in F$. Writing $\beta = \sum_{i\in I} b_i \eps_i$ with $b_i\geq 0$ we find
\[ (\beta,\beta) = \sum_{i\in I}b_i(\beta,\eps_i) \leq 0 \]
since by definition $(\beta,\eps_i)\leq 0$ for each $i\in I$. 
\end{proof}

We may now state the key result on the existence and uniqueness of simple
$\Pi^\lambda$-modules. In the following result the function
$p\colon \Z^I\ra \Z$ is defined by the formula $p(\alpha)=1-\frac12 (\alpha,\alpha)$. 

\begin{theorem}\cite[Theorem~1.2]{Crawley01} \label{CB-Theorem}
Let $Q$ be a quiver
with vertex set $I$. Let
$\Pi^\lambda$ be the associated deformed preprojective algebra of weight
$\lambda\in \Z^I$. Then for each $\alpha\in \Z^I_{\geq 0}$ the following is
equivalent
\begin{enumerate}[label = (\roman*), ref = (\roman*)]
\item There exists a simple $\Pi^\lambda$-module of dimension $\alpha$.
\item $\alpha$ is a positive root with $\lambda\cdot \alpha=0$ and for
every decomposition $\alpha =\sum_t \beta^{(t)}$ into positive roots
satisfying $\lambda\cdot \beta^{(t)}=0$ one has
\[ p(\alpha)> \sum_t p(\beta^{(t)})  .\]
\end{enumerate}
In that situation $\mu_\C^{-1}(\lambda)\subset \om{Rep}(\ovl{Q},\alpha)$
is a reduced and irreducible complete intersection of dimension
$\alpha\cdot \alpha -1 +2p(\alpha)$ and the general element is a simple
representation. 
\end{theorem}

\section{Extended Dynkin Quivers and their Root Systems}
In our later work it will be important to have a firm grip on the
relation between the Dynkin diagrams and root systems of type $ADE$ and
their extended counterparts of type $\wt{ADE}$. In this section we
briefly review the necessary root space theory, establish our notation
and prove two basic lemmas needed to effectively apply Theorem
\ref{CB-Theorem}.   

Let $K$ be a Dynkin diagram of type $A_n$, $D_n$, $E_6$, $E_7$ or $E_8$,
for short type $ADE$. Fix an identification of the set of vertices
with $\{1,2,\cdots,n\}$ for some $n\in \N$. The Cartan
matrix $C=(c_{ij})_{ij} \in M_n(\Z)$ of $K$ is then defined by $c_{ij}=2\delta_{ij}-a_{ij}$, where $a_{ij}=a_{ji}=1$ precisely when there is an edge connecting $i$ to $j$ in $K$ and otherwise $0$. The associated root system 
$\Phi\subset \Z^n$
is then constructed just as in the previous section using the pairing
$(v,w)_C \coloneqq  v\cdot Cw$. Note that this pairing is positive definite so in particular $\Phi=\Phi^{\om{re}}$. In particular, the coordinate
vectors $\{\eps_i :1\leq i\leq n\}$ is a set of simple roots for $\Phi$.
There is a unique maximal root $d\in \Phi^+\subset \Z^n$ with respect to the partial ordering $\leq$ on $\Z^n$ (see \cite[Section~10.4]{Humphreys78}). 
The extended Dynkin diagram 
$\wt{K}$ is constructed from $K$ by adjoining a single vertex $0$ and one edge connecting $0$ to $i$ if $(d,\eps_i)_C = 1$ for each $1\leq i\leq n$. The extended Cartan matrix $\tilde{C}$ is constructed from $\wt{K}$ in the same way $C$ was constructed from $K$. Explicitly, if we identify $\Z^{n+1}=\Z \eps_0\oplus \Z^n$,
\[ \wt{C} = \left( \begin{array}{cc} 2 & -d^tC \\ -Cd & C \end{array}\right). \]
The associated root system $\wt{\Phi}\subset \Z^{n+1}$ is then constructed
using the pairing $(v,w)_{\wt{C}}\coloneqq v^t\wt{C}w$. We have the following
useful description of the real roots in $\Phi$ and $\wt{\Phi}$
(see \cite[Proposition~5.10]{Kac90})
\begin{equation} \label{Real-Root-Criterion}
\Phi = \{ \alpha \in \Z^n :(\alpha,\alpha)_C =2\} \;\; \mbox{ and } \;\;
 \wt{\Phi}^{re} = \{\beta\in \Z^{n+1} :(\beta,\beta)_{\wt{C}}=2 \} .
\end{equation} 
To understand the imaginary roots in $\wt{\Phi}$ define a linear map
$\psi\colon \Z^{n+1}\ra \Z^n$ by $\psi(\eps_0)=-d$ and $\psi(\eps_i)=\eps_i$ for $1\leq i\leq n$. Then, using the above explicit description of $\wt{C}$, one obtains the following identity
\[ (v,w)_{\wt{C}} =(\psi(v),\psi(w))_C  .\]
As the latter pairing is positive definite one deduces that 
$(\cdot,\cdot)_{\wt{C}}$ is positive semi-definite. It follows by Lemma \ref{RootLemma0} that the set of imaginary roots must coincide with the
nonzero elements of $\om{Ker}(\psi)$, that is,
\[ \wt{\Phi}^{\om{im}}=\{r\delta :r\in \Z-\{0\}\}   \]
where $\delta = (1,d)^t\in \Z\eps_0\oplus \Z^n = \Z^{n+1}$ is the minimal positive imaginary root. 
We will need two lemmas concerning these root systems.

\begin{lemma} \label{Root-Bijection-Lemma} 
Define $\Sigma \coloneqq \{\beta \in \wt{\Phi} :0<\beta<\delta\}$.
Then the map $\psi\colon\Z^{n+1}\ra \Z^n$ restricts to a bijection
$\psi\colon \Sigma \ra \Phi$ with inverse given by
\[ \psi^{-1}(\alpha) = \left\{ \begin{array}{cc} (0,\alpha) & \mbox{ if } \alpha\in \Phi^+ \\
(1,d+\alpha) & \mbox{ if } \alpha\in \Phi^- \end{array} \right. \]
with respect to the decomposition $\Z^{n+1}=\Z\eps_0 \oplus \Z^n$.
Furthermore, the adjoint $\psi^*\colon \R^n\ra \R^{n+1}$, determined by
$\psi(\theta)\cdot \tau = \theta \cdot \psi^*(\tau)$ for
$\theta\in \Z^{n+1}$ and $\tau\in \R^n$, is given by
$\psi^*(\tau) = (-d\cdot \tau,\tau)$ and corestricts to an isomorphism
$\R^n\cong \delta^\perp \subset \R^{n+1}$. 
\end{lemma}
\begin{proof} Note first $\Sigma\subset \wt{\Phi}^{\om{re}}$ since
$\delta$ is the minimal positive imaginary root. As
$(\alpha,\beta)_{\wt{C}}=(\psi(\alpha),\psi(\beta))_{C}$ for all
$\alpha,\beta\in \Z^{n+1}$, it follows
from the description of the real roots in \eqref{Real-Root-Criterion} that
$\psi(\Sigma)\subset \Phi$. The same result shows that the map
$\kappa\colon \Phi\ra \Sigma$ given by $\kappa(\alpha)=(0,\alpha)$ if $\alpha\in \Phi^+$ and $\kappa(\beta)=(1,d+\beta)$ if $\beta\in \Phi^-$ is well-defined. Using the definition of $\psi$ one easily verifies that
$\psi\kappa=\om{id}_{\Phi}$ and $\kappa\psi = \om{id}_{\Sigma}$. Hence,
$\psi$ is a bijection with inverse $\psi^{-1}=\kappa$.

For the second part note that $\psi$ extends uniquely to a linear map
$\psi\colon \R^{n+1}\ra \R^n$ and hence has an adjoint 
$\psi^*:\R^n\ra \R^{n+1}$ uniquely determined by the formula given in the statement. For each $1\leq i \leq n$
we find $\psi^*(\tau)_i = \psi^*(\tau)\cdot \eps_i = \tau\cdot \psi(\eps_i)=\tau_i$, while $\psi^*(\tau_0) = \tau\cdot \psi(\eps_0)=-\tau\cdot d$. Thus $\psi^*(\tau) = (-d\cdot \tau,\tau)$. Finally, 
since $\psi:\R^{n+1}\ra \R^n$ is surjective, it follows that $\psi^*$ corestricts to an isomorphism onto $\om{Ker}(\psi)^\perp = \delta^\perp$.    
\end{proof}

In the following lemma we regard $\Phi\subset \Z^n\subset \R^n$ as above,
and we write $(\cdot,\cdot)\colon \Z^n\times \Z^n\ra \Z$ for the Cartan
pairing. 

\begin{lemma} \label{RootSpace-Decomp-Lemma} 
For $\tau \in \C^n$ define $\tau^\perp \coloneqq \om{Span}_\R(\om{Re}\tau,\om{Im}\tau)^\perp\subset \R^n$ with respect to the standard
scalar product on $\R^n$. Then if $\tau^\perp\cap \Phi$
is non-empty, it is a root system in the subspace it spans and decomposes into a disjoint union of root systems of type $ADE$
\[ \tau^\perp\cap \Phi = \Phi_1\cup \Phi_2\cup \cdots \cup \Phi_r .\]
Furthermore, $\Phi_j$ admits a unique base contained in
$\Phi^+$ for each $1\leq j\leq r$. 
\end{lemma}
\begin{proof} Write $\Phi_\tau=\tau^\perp\cap \Phi$. The fact that 
$\tau^\perp\cap \Phi$ is a root system in the subspace it spans follows
from \cite[Exercise~III.9.7]{Humphreys78}. To see that $\Phi_\tau$
admits a base contained in $\Phi^+\subset \Z^n_{\geq 0}$ we mimic the proof
for the existence of bases in a root system in \cite[p.~48]{Humphreys78}.  
We may write $\Phi_\tau = \Phi_\tau^+\cup \Phi_\tau^-$ where
$\Phi_\tau^{\pm} = \tau^\perp \cap \Phi^\pm$. As $\Phi_\tau^- = -\Phi_\tau^+$, it follows that $\Phi_\tau$ is nonempty if and only if
$\Phi_\tau^+$ is nonempty. We may therefore define $S\subset \Phi_\tau^+$ to be the subset of $\alpha\in \Phi_\tau^+$ that admits
no decomposition $\alpha = \beta +\gamma$ for $\beta,\gamma\in \Phi_\tau^+$. This set is nonempty since any
$\alpha =\sum_i a_i\eps_i\in \Phi_\tau^+$, $a_i \geq 0$, with $\sum_i a_i$ minimal must belong to $S$.
For any pair $\alpha\neq\beta \in S$ we have have $(\alpha,\beta)\leq 0$.
Indeed, if $(\alpha,\beta)=1$, then either $\alpha-\beta$ or $\beta-\alpha$ will belong to $\Phi_\tau^+$ contradicting either $\alpha\in S$ or $\beta\in S$. To see that the set $S$ is linearly independent,
suppose that $\sum_{s\in S} a_ss=0$. Put $S_1 = \{s\in S:a_s>0\}$, 
$S_2 = S-S_1$ and write $u = \sum_{s\in S_1} a_ss = \sum_{t\in S_2} b_tt$
where $a_s>0$ and $b_t = -a_t\geq 0$. Then
\[ (u,u)=  \sum_{s,t} a_sb_t(s,t)\leq 0,  \]
which is only possible if $u = 0$. Hence, as each $s\in S$ is nonzero
and has non-negative coefficients with respect to the standard basis
$\eps_i$, $1\leq i\leq n$, it follows that $a_s=0$ for all $s\in S$
as required. It is clear that every root $\alpha\in \Phi_\tau^+$
can be written as a positive integral linear combination of the
elements of $S$ and we have thus verified that $S$ is a base for $\Phi_\tau$. 
At this point we may decompose $S=S_1\cup S_2\cup \cdots \cup S_r$
into pairwise orthogonal sets in such a way that each $S_i$ is indecomposable, i.e., admits no further decomposition into pairwise
orthogonal sets. This yields a corresponding 
decomposition into irreducible root systems (see \cite[Section~10.4]{Humphreys78}) 
$\Phi_\tau = \Phi_1\cup \Phi_2\cdots \cup \Phi_r,$
where $S_i$ is a base for $\Phi_i$ for each $1\leq i\leq r$. As each
$\Phi_j$ is contained in $\Phi$, all the roots have the same length
and this implies that $\Phi_j$ must be of type $ADE$ for each $j$. 
\end{proof}

The graphs $K$ and $\wt{K}$ are transformed into quivers by giving
the edges arbitrary orientations. As already mentioned the corresponding
symmetric bilinear forms and root systems are independent of the
choice of orientations. In particular, if $Q$ is an extended Dynkin
quiver we may identify the set of vertices with $\{0,1,\cdots,n\}$
for some $n\in \N$ and assume that we have root systems $\wt{\Phi}\subset \Z^{n+1}$,
$\Phi\subset \Z^n$ such that the minimal positive imaginary root
$\delta$ takes the form $(1,d)\in \Z\eps_0\oplus \Z^n$, where
$d\in \Phi$ is the maximal positive root. Furthermore, by Lemma \ref{Root-Bijection-Lemma} we have the
map $\psi\colon \Z^{n+1}\ra \Z^n$ relating them and the adjoint $\psi^*:\R^n\ra \R^{n+1}$ that allows us to identify $\R^n\cong \delta^\perp$. We will work under these assumptions whenever convenient in the rest of the paper.  
 
\section{Classification of Singularities}
Let $Q$ be an extended Dynkin quiver with vertex set $I$ and minimal
imaginary root $\delta\in \Z^I$. In this section we will give a
description of the singular set in the quiver variety 
$\MC{M}_{(0,\lambda)}(Q,\delta)$ for $\lambda\in \C^I$.
According to Proposition \ref{Semi-Simp-Quiver-Eq-Prop} the singular set
$\MC{M}^{\om{sing}}_{(0,\lambda)}(Q,\delta)$ is in natural bijection with
the isomorphism classes of semi-simple, non-simple $\Pi^\lambda$-modules
of dimension $\delta$, so it suffices to determine the latter set. 

For this purpose let $\wt{\Phi}$ denote the root system associated with
$Q$ and let $\Sigma=\{\alpha \in \wt{\Phi}:0<\alpha<\delta\}$ as in Lemma \ref{Root-Bijection-Lemma}. For $\lambda\in \C^I$ define $\Sigma_\lambda = \{\alpha\in \Sigma:\alpha\cdot \lambda=0\}$ and let this set be partially ordered by $\alpha\prec \beta$ if and only if
$\beta-\alpha = \sum_t \gamma^{(t)}$ for some $\gamma^{(t)}\in \Sigma_\lambda$.
Finally, let $\Sigma_\lambda^{\om{min}}\subset \Sigma_\lambda$ denote the subset of
minimal elements with respect to this partial ordering. 

\begin{lemma} \label{Simp-Existence-Lemma} 
There exists a simple $\Pi^\lambda$-module of dimension
$\delta$ if and only if $\delta\cdot \lambda=0$. Moreover, there exists
a simple $\Pi^\lambda$-module of dimension $\alpha$ satisfying $0<\alpha<\delta$ if and only if $\alpha\in \Sigma_\lambda^{\om{min}}$ and in that case the simple module
is unique up to isomorphism. 
\end{lemma}
\begin{proof} According to Theorem \ref{CB-Theorem} there exists a simple
$\Pi^\lambda$-module of dimension $\alpha \in \Z^I_{\geq 0}$ if and only if
$\alpha$ is a root satisfying $\alpha\cdot \lambda=0$ and for every decomposition $\alpha = \sum_t \beta^{(t)}$ into positive roots satisfying 
$\beta^{(t)}\cdot \lambda=0$, it holds true that $p(\alpha)>\sum_t p(\beta^{(t)})$, where we recall that $p(\alpha)=1-\frac12 (\alpha,\alpha)$. 
In our case of an extended Dynkin quiver we have $p(\delta)=1-\frac12 (\delta,\delta)=1$ and $p(\alpha)=1-\frac12 (\alpha,\alpha)=0$ for every real root
$\alpha \in \wt{\Phi}^{\om{re}}$. In any decomposition $\delta = \sum_t \beta^{(t)}$
into positive roots with at least two summands the roots $\beta^{(t)}$ must be real because $\delta$ is the minimal positive imaginary root. Therefore, the
condition $p(\delta)=1>0=\sum_t p(\beta^{(t)}$ is trivially satisfied. We
conclude that there exists a simple $\Pi^\lambda$-module of dimension 
$\delta$ if and only if $\delta\cdot \lambda=0$.
 
If $\alpha$ satisfies $0<\alpha <\delta$, there exists a simple $\Pi^\lambda$-module of dimension $\alpha$ if and only if $\alpha\in \Sigma_\lambda$ and
for every decomposition $\alpha=\sum_t \beta^{(t)}$ with $\beta^{(t)}\in \Sigma_\lambda$ it holds true that $p(\alpha)>\sum_t p(\beta^{(t)})$.
This inequality is never satisfied since both sides reduce to
zero. Consequently, the above condition can only be satisfied if
$\alpha$ does not admit such a decomposition at all and this is equivalent
to $\alpha\in \Sigma^{\om{min}}_\lambda$. The fact that the simple 
$\Pi^\lambda$-module is unique up to isomorphism in this case follows from the final part of Theorem \ref{CB-Theorem} as explained in \cite[p.~260]{Crawley01}.
\end{proof}    

Before we proceed we record the following consequence.

\begin{lemma} \label{Connectivity}
Let $Q$ be an extended Dynkin quiver with vertex set $I$
and minimal imaginary root $\delta$. Let $\lambda\in \C^I$ satisfy
$\lambda\cdot \delta=0$.  Then the quiver variety $\MC{M}_{(0,\lambda)}(Q,\delta)$ is connected and $\MC{M}_{(0,\lambda)}^{\om{reg}}(Q,\delta)$ is nonempty. \end{lemma}
\begin{proof} By the above lemma there exists a simple $\Pi^\lambda$-module
of dimension $\delta$ in this situation. By Proposition \ref{Semi-Simp-Quiver-Eq-Prop} this implies that $\MC{M}_{(0,\lambda)}^{\om{reg}}(Q,\delta)$ is
nonempty. Furthermore, by the final part of Theorem \ref{CB-Theorem} the
variety $\mu_\C^{-1}(\lambda)$ is irreducible in the Zariski topology.
It is therefore connected in the analytic topology and it follows that
the quotient $\MC{M}_{(0,\lambda)}(Q,\delta)\cong \mu_\C^{-1}(\lambda)//G_\delta^c$ is connected as well. \end{proof} 

In the following theorem we make the assumptions on the extended
Dynkin quiver $Q$ as explained in the end of the previous section.

\begin{theorem} \label{SemiSimpClass}
Let $Q$ be an extended Dynkin quiver with vertex set $\{0,1,\cdots,n\}$ and
let $\Pi^\lambda$ be the associated deformed preprojective algebra
of weight $\lambda\in \C^{n+1}$ satisfying $\lambda\cdot \delta=0$. 
Let $\Phi\subset \Z^n$ be the root system of type $ADE$ associated with
$Q$. Write $\lambda = (\lambda_1,\tau)$ where $\lambda_1\in \C$ and 
$\tau\in \C^n$ and let
\[ \tau^\perp \cap \Phi = \Phi_1\cup \cdots \cup \Phi_r  \]
be a decomposition into (irreducible) subsystems of type $ADE$ as in
Lemma \ref{RootSpace-Decomp-Lemma}. Then there is
a bijection between $\{\Phi_1,\cdots,\Phi_r\}$ and the isomorphism classes
of semi-simple, non-simple $\Pi^\lambda$-modules
of dimension $\delta$. \end{theorem}
\begin{proof} Let $\wt{\Phi}\subset \Z^{n+1}$ be the root system 
associated with $Q$. Let $\Sigma_\lambda^{\om{min}}\subset \Sigma_\lambda \subset \Sigma\subset \wt{\Phi}$ be defined as in the beginning of the section.
The content of Lemma \ref{Simp-Existence-Lemma} is then that there
exists a simple $\Pi^\lambda$-module of dimension $\alpha$, $0<\alpha<\delta$
if and only if $\alpha\in \Sigma_\lambda^{\om{min}}$ and in that case the
module is unique up to isomorphism. This implies that 
a semi-simple, non-simple $\Pi^\lambda$-module $X=\sum_{t=0}^k n_tX_t$ of dimension $\delta$ is uniquely determined up to isomorphism by the
roots $\gamma_t\coloneqq  \om{dim}X_t\in \Sigma_\lambda^{\om{min}}$ and the 
multiplicities $n_t\in \N$. 
We therefore have a bijective correspondence between the isomorphism classes
of semi-simple, non-simple $\Pi^\lambda$-modules of dimension $\delta$ and sets
$\{(n_t,\gamma_t)\}_{t=0}^k$ for which $n_t\in \N$,
$\gamma_t\in \Sigma_\lambda^{\om{min}}$ for each $t$, $\delta=\sum_t n_t\gamma_t$
and either $k\geq 1$ or $n_0>1$.

Our task is to relate the collection of such sets with the root systems
in the decomposition
\[ \tau^\perp\cap \Phi=\Phi_1\cup \cdots \cup \Phi_r \]
given in the statement of the theorem. Suppose that $\{(n_t,\gamma_t)\}_{t=0}^k$
is such a set. As $\delta = (1,d)\in \Z\oplus \Z^n$, where $d\in \Phi$ is the
maximal root, the condition $\delta=\sum_t n_t\gamma_t$ implies that there
is a distinguished root $\gamma_t$ with nonzero first component and thus necessarily $n_t=1$.
After possibly rearranging the roots we may take this root to be $\gamma_0$. 
By Lemma \ref{Root-Bijection-Lemma} there are unique positive roots 
$\beta,\alpha_t\in \Phi^+$, $1\leq t\leq k$, such that 
$\gamma_0 =\psi^{-1}(-\beta)= (1,d-\beta)$ and $\gamma_t = \psi^{-1}(\alpha_t)= (0,\alpha_t)$ for $1\leq t\leq k$. Moreover, since $\lambda\cdot \delta=0$,
there is a unique $\tau\in \C^n$ such that $\lambda = (-d\cdot \tau,\tau)=\psi^*(\tau)$. The relation $\theta\cdot \lambda = \psi(\theta)\cdot \tau$ for each $\theta\in \Z^{n+1}$ ensures that the bijection
$\psi\colon \Sigma \cong \Phi$ restricts to a bijection
$\Sigma_\lambda \cong \Phi\cap \tau^\perp$. In particular,
$\beta,\alpha_1,\cdots,\alpha_k\in \Phi\cap \tau^\perp$.  
Moreover, the minimality of $\gamma_t=(0,\alpha_t)$, $1\leq t\leq k$, translates to the fact that each $\alpha_t$ is minimal among the
roots in $\Phi^+\cap \tau^\perp$, while the minimality of 
$\gamma_0=(1,d-\beta)$ translates to the fact that 
$\beta\in \Phi^+\cap \tau^\perp$ is maximal.
This means that $\beta$ must be the unique maximal positive root in precisely one of the systems $\Phi_j$ occurring in the decomposition of $\Phi\cap \tau^\perp$. Furthermore, since the equality $\delta = \sum_t n_t\gamma_t$
is equivalent to the equality $\beta=\sum_t n_t\alpha_t$, we also
deduce that $\{\alpha_t:1\leq t\leq k\} $ must be the unique positive base in the same system. 

This procedure is clearly reversible. Given a system $\Phi_j$ let
$\alpha_t$, $1\leq t\leq k$ be the unique positive base and let
$\beta = \sum_t n_t\alpha_t$ be the maximal root. We may then define
$\gamma_0 = (1,d-\beta)\in \Sigma_\lambda^{\om{min}}$, $n_0=1$ and 
$\gamma_t = (0,\alpha_t)\in \Sigma_\lambda^{\om{min}}$ for $1\leq t\leq k$. It
then follows from our previous arguments that the set $\{(n_t,\gamma_t\}_{t=1}^k$ satisfies the required conditions:
$n_t\in \N$, $\gamma_t\in \Sigma_\lambda^{\om{min}}$ for all $t$ and $\sum_t n_t\gamma_t=\delta$. 
This completes the proof of the theorem. 
\end{proof}
 
\section{Local Structure and the Proof of Theorem \ref{Singularity-Theorem}}
The combination of Proposition \ref{Semi-Simp-Quiver-Eq-Prop} and
Theorem \ref{SemiSimpClass} give full control over the singularities
in $\MC{M}_{(0,\lambda)}(Q,\delta)$ for an extended Dynkin quiver $Q$.
In this section we establish the final results needed to complete
the proof of Theorem \ref{Singularity-Theorem}.  

Let $Q$ be a quiver with vertex set $I$
and let $\lambda\in \C^I$ be a parameter. Given a point $x\in \mu^{-1}(\lambda)$ consider the sequence
\[ \begin{tikzcd} G^c_v \arrow{r}{b_x} & \om{Rep}(\ovl{Q},\delta)
\arrow{r}{\mu_\C} & \fg_v^c ,\end{tikzcd}  \]
where $b_x(g) = g\cdot x$ is the orbit map at $x$. As $\mu_\C$ is
$G^c_v$-equivariant and $\lambda\in \C^I$ is identified with an
element of $Z(\fg^c_v)$, the composition $\mu_\C\circ b_x$ is the
constant map at $\lambda$. Hence, by differentiating this sequence at
$1\in G^c_v$ we obtain a three term complex
\begin{equation} \label{Complex1}
\begin{tikzcd} 0\arrow{r} & \fg_v^c \arrow{r}{\sigma_x} & \om{Rep}(\ovl{Q},v)
\arrow{r}{\nu_x} & \fg_v^c \arrow{r} & 0 \end{tikzcd}
\end{equation}    
where $\sigma_x = d(b_x)_1$ and $\nu_x = d(\mu_\C)_x$. 
If $x\in \mu_\R^{-1}(0)$ such that the orbit $G_v^c\cdot x$ is closed
and hence an embedded complex submanifold of $\om{Rep}(\ovl{Q},v)$, one may identify $\om{Im}(\sigma_x)=T_x(G^c_v\cdot x)$. By general properties of the
moment map it holds true that $\om{Ker}(\nu_x) = \om{Im}(\sigma_x)^{\omega_\C}$, where the upper case $\omega_\C$ denotes the complex symplectic complement.
In particular, the space $T_x(G^c_v\cdot x)$ is isotropic with respect to $\omega_\C$. Moreover, the stabilizer $H \coloneqq (G^c_v)_x$ acts linearly on all the spaces involved and the maps $\sigma_x$ and $\nu_x$ are $H$-equivariant. Therefore, $T_x(G^c_v \cdot x)^{\omega_\C}/T_x(G^c_v\cdot x)= \om{Ker}(\nu_x)/\om{Im}(\sigma_x)$ obtains a complex symplectic
form preserved by the induced action of $H$. 

\begin{definition} \label{Def-Symp-Slice}
Let $x\in \mu^{-1}(0,\lambda)$. Then the complex
symplectic slice at $x$ is the complex symplectic $(G_v^c)_x$-representation 
\[ T_x(G^c_v\cdot x)^{\omega_\C}/T_x(G^c_v\cdot x) = \om{Ker}(\nu_x)/\om{Im}(\sigma_x)  .\]
\end{definition}  

The following result is a consequence of \cite[Theorem~1.4(iv)]{Mayrand18}.
Here we regard $\MC{M}_{(0,\lambda)}(Q,v)$ as a complex analytic space
using Theorem \ref{Quotient-Equivalence-Theorem}. 

\begin{lemma} \label{Local-Structure-Lemma} 
Let $Q$ be a quiver with vertex set $I$, let $v\in \Z^I$
be a dimension vector and let $\lambda\in \C^I$ be a parameter. 
Let $y\in \MC{M}_{(0,\lambda)}(Q,v)$ and let
$x\in \mu^{-1}(0,\lambda)\subset \om{Rep}(\ovl{Q},\delta)$ be a point
above $y$. Set
\[ H \coloneqq (G^\C_v)_x  \;\; \mbox{ and } \;\;  
W \coloneqq T_p(G^c_v\cdot x)^{\omega_\C}/T_x(G^c_v\cdot x)  \]
Let $\mu_W\colon W\ra \mathfrak{h}^*$ be the
unique complex symplectic moment map vanishing at $0$, where $\mathfrak{h}=\om{Lie}(H)$. Then a neighborhood of $y\in \MC{M}_{(0,\lambda)}(Q,v)$ is biholomorphic with a neighborhood of $0$ in (the analytification of) the GIT quotient $\mu_W^{-1}(0)//H$.
\end{lemma}

In view of this result our task is to determine the complex
symplectic slices at the points above the singular points in
$\MC{M}_{(0,\lambda)}(Q,\delta)$. It will be useful to introduce the
following notation. 

\begin{definition} Let $Q$ be a quiver with vertex set $I$. For a pair
of dimension vectors $v,w\in \Z^I_{\geq 0}$ define
\[ \om{Hom}(v,w)\coloneqq \bigoplus_{i\in I} \om{Hom}(V_i,W_i) \;\; \mbox{ and }
\;\; \om{Rep}(Q;v,w) \coloneqq \bigoplus_{h\in \ovl{Q}}\om{Hom}(V_{s(h)},W_{t(h)}), \]
where $V_i = \C^{v_i}$ and $W_i=\C^{w_i}$ for each $i\in I$.  
\end{definition}

Note that $\om{Rep}(Q;v,v)=\om{Rep}(Q,v)$ and that $\om{End}(v) \coloneqq \om{Hom}(v,v)=\fg_v^c$. The complex in \eqref{Complex1} also has a relative
analogue. Let $v,w\in \Z^I$ be a pair of dimension vectors and let
$x\in \om{Rep}(\ovl{Q},v)$ and $y\in \om{Rep}(\ovl{Q},w)$ satisfy
$\mu_\C(x)=\mu_\C(y)=\lambda$ for some $\lambda\in \C^I$. Define
$C_Q(x,y)$ to be the sequence given by
\[ \begin{tikzcd} 0 \arrow{r} & \om{Hom}(v,w) \arrow{r}{\sigma_{x,y}}
& \om{Rep}(Q;v,w) \arrow{r}{\nu_{x,y}} & \om{Hom}(v,w) \arrow{r} & 0
\end{tikzcd} \]
where 
\begin{align*}
\sigma_{x,y}((u_i)_{i\in I}) &= \; (u_{t(h)}x_h-y_hu_{s(h)})_{h\in \ovl{Q}} \\
\nu_{x,y}((v_h)_{h\in \ovl{Q}}) &= \; \left(\sum_{h\in t^{-1}(i)}\eps(h)(u_hx_{\ovl{h}}+y_hu_{\ovl{h}}) \right)_{i\in I}   .
\end{align*}
Note that $C_Q(x,x)$ is the complex of \eqref{Complex1}. 

\begin{lemma} \label{Homology-Lemma} 
Let $X$ and $Y$ denote the $\Pi^\lambda$-modules corresponding
to $x\in \om{Rep}(\ovl{Q},v)$ and $y\in \om{Rep}(\ovl{Q},w)$. Then
$C_Q(x,y)$ is a chain complex, i.e., $\nu_{x,y}\circ \sigma_{x,y}=0$ and
if we denote the cohomology groups from left to right by
$H^i_Q(x,y)$ for $0\leq i\leq 2$ we have
\begin{enumerate}[label=(\arabic*), ref=(\arabic*)]
\item $H^0_Q(x,y) \cong \om{Hom}_{\Pi^\lambda}(X,Y)$,
\item $H^2_Q(x,y) \cong \om{Hom}_{\Pi^\lambda}(Y,X)^*$,
\item $\om{dim}_\C H^1_Q(x,y) =\om{dim}_\C H^0_Q(x,y)+\om{dim}_\C H^2_Q(x,y)
-(v,w)$. 
\end{enumerate} 
\end{lemma} 
\begin{proof} To simplify the notation we will write $V_i = \C^{v_i}$ and
$W_i = \C^{w_i}$ for $i\in I$. Let $u = (u_i\colon V_i\ra W_i)_{i\in I} \in \om{Hom}(v,w)$. Then using the definitions of $\sigma_{x,y}$ and $\nu_{x,y}$
we see that $\nu_{x,y}\circ \sigma_{x,y}(u)$ equals
\begin{align*} 
 =& \left( \sum_{h\in t^{-1}(i)} 
\eps(h)(u_{t(h)}x_hx_{\ovl{h}}-y_hu_{s(h)}x_{\ovl{h}}+y_hu_{t(\ovl{h})}x_{\ovl{h}}
-y_hy_{\ovl{h}}u_{s(\ovl{h})} ) \right)_{i\in I}  \\
=& \left( \sum_{h\in t^{-1}(h)} u_i(\eps(h)x_hx_{\ovl{h}}) -(\eps(h)y_hy_{\ovl{h}})u_i \right)_{i\in I} = (u_i\lambda_i - \lambda_i u_i)_{i\in I} =0 . 
\end{align*}
Here we have used that $s(\ovl{h})=t(h)$, $t(\ovl{h})=s(h)$ and that
$\mu_\C(x)=\mu_\C(y) = \lambda$. This shows that $C_Q(x,y)$ is a chain
complex.   

Recall that $\Pi^\lambda$ was defined to be a quotient
of the quiver algebra $\C \ovl{Q}$. Therefore, we may also
regard $X$ and $Y$ as $\C \ovl{Q}$-modules and clearly 
$\om{Hom}_{\C \ovl{Q}}(X,Y) = \om{Hom}_{\Pi^\lambda}(X,Y)$.
From the definition of a homomorphism of representations it is clear
that $\om{Hom}_{\C \ovl{Q}}(X,Y)=\om{Ker}(\sigma_{x,y}) = H^0_Q(x,y)$ proving
part $(1)$. 

For the second part we use an idea from the proof of  \cite[Lemma~3.1]{Crawley00} (this lemma and its proof implies our result for $\lambda=0$).
Let $\phi\colon \om{Hom}(w,v)\ra \om{Hom}(v,w)^*$ be the isomorphism given by
$\phi(u)(v) = \sum_{i\in I} \om{tr}(u_iv_i)$ and let $\psi\colon \om{Rep}(Q;w,v)\ra \om{Rep}(Q;v,w)^*$ be the isomorphism given by
$\psi(f)(g) = \sum_{h\in \ovl{Q}}\eps(h)\om{tr}(f_hg_{\ovl{h}})$. 
Then a rather tedious calculation shows that the following diagram commutes 
\[ \begin{tikzcd} \om{Hom}(w,v) \arrow{r}{\sigma_{y,x}} \arrow{d}{\phi}
& \om{Rep}(Q;w,v) \arrow{d}{\psi} \\
\om{Hom}(v,w)^* \arrow{r}{(\nu_{x,y})^*} & \om{Rep}(Q;v,w)^*.
\end{tikzcd} \]  
Since both the vertical maps are isomorphisms we conclude that
\[ \om{Coker}(\nu_{x,y})^* \cong \om{Ker}((\nu_{x,y})^*)\cong \om{Ker}(\sigma_{y,x}) =\om{Hom}_{\Pi^\lambda}(Y,X) , \]
where the final equality follows from the first part. Hence, $H^2_Q(x,y)\cong \om{Hom}_{\Pi^\lambda}(Y,X)^*$. 

For the final part observe that
\[ (v,w) = 2\sum_{i\in I}v_iw_i -\sum_{h\in \ovl{Q}} v_{s(h)}w_{t(h)}
= 2\om{dim}_\C \om{Hom}(v,w) - \om{dim}_\C \om{Rep}(Q;v,w)  \]
is the Euler characteristic of the complex $C_Q(v,w)$. Since the
Euler characteristic is preserved upon passage to cohomology, we obtain
$(v,w) = \om{dim}_\C H^0_Q(v,w)-\om{dim}_\C H^1_Q(v,w) + \om{dim}_\C H^2_Q(v,w)$ and this is equivalent to the formula stated in part $(3)$.   
\end{proof}

\begin{remark} It is in fact also true that $H^1_Q(x,y) \cong \om{Ext}^1_{\Pi^\lambda}(X,Y)$. We give a sketch of the proof. By 
\cite[Corollary~1.4.2]{Brion12} it holds true that
$\om{Coker}(\sigma_{x,y}) = \om{Ext}^1_{\C \ovl{Q}}(X,Y)$.
Moreover, there is an explicit way to relate this group to the set of
isomorphism classes of extensions $0\ra Y\ra Z \ra X \ra 0$.
Given an element $[z]\in \om{Ext}^1_{\C \ovl{Q}}(X,Y)$ represented
by $z = (z_u\colon V_{s(h)}\ra W_{t(h)})$ one may construct the extension $Z$
by setting $e_iZ=U_i = V_i\oplus W_i$ for each $i\in I$ and letting
$z_h\colon U_{s(h)}\ra U_{t(h)}$ for $h\in \ovl{Q}$ be given by the matrix
\[ z_h =\left(  \begin{array}{cc} x_h & 0 \\ z_h & y_h \end{array} \right) .\]
The exact sequence $0\ra Y\ra Z\ra X\ra 0$ is given componentwise by the canonical exact sequence $0\ra W_i\ra V_i\oplus W_i\ra V_i\ra 0$.
This is then an extension of $\Pi^\lambda$-modules if and only if
$\mu_\C(Z)=\lambda$. It is then a matter of calculation to check that this
is the case if and only if $z\in \om{Ker}(\nu_{x,y})$. 
\end{remark}  

Let $Q$ be an extended Dynkin quiver with vertex set identified with
$\{0,1,\cdots,n\}$ and minimal imaginary root $\delta\in \Z^{n+1}$.
Let $\lambda=(\lambda_1,\tau) \in \C\oplus \C^n=\C^{n+1}$ satisfy
$\delta\cdot \lambda=0$. Denote the root systems by $\tilde{\Phi}\subset \Z^{n+1}$ and $\Phi\subset \Z^n$ as usual. By Proposition \ref{Semi-Simp-Quiver-Eq-Prop} and Theorem \ref{SemiSimpClass} the singular points in
$\MC{M}_{(0,\lambda)}(Q,\delta)$ are in bijection with the components
in the root space decomposition
\[ \Phi\cap \tau^\perp = \Phi_1\cup \Phi_2 \cup \cdots \cup \Phi_r  .\]
Write $\MC{M}^{\om{sing}}_{(0,\lambda)}(Q,\delta) = \{y_1,y_2,\cdots,y_r\}$ where $y_i$ corresponds to $\Phi_i$ for each $1\leq i \leq r$. 

\begin{proposition} \label{Symplectic Slice} 
In the above situation fix $i$, $1\leq i\leq r$, and
let $x\in \mu^{-1}(0,\lambda)\subset \om{Rep}(Q,\delta)$ be a point above $y_i$. Let $Q'$ be the extended Dynkin quiver associated with the root system
$\Phi_i$ and let $\delta'$ denote its minimal imaginary root. Then there
is an isomorphism $(G_\delta^c)_x \cong G_{\delta'}^c$ and there
is a complex symplectic isomorphism
\[ T_x(G^c_\delta\cdot x)^{\omega_\C}/T_x(G^c_\delta\cdot x) \cong \om{Rep}(\ovl{Q'},\delta')   \]
equivariant along the above isomorphism of groups. 
\end{proposition}
\begin{proof} First note that the complex symplectic slice at $x$
is precisely the cohomology group $H^1_Q(x,x)$. We will determine this
complex symplectic space as an $H\coloneqq (G^c_\delta)_x$ representation.
Let $X=\oplus_{t=0}^k n_tZ_t$ denote the semi-simple $\Pi^\lambda$-module corresponding to $x$ decomposed into simple summands.
Then, according to Proposition \ref{Semi-Simp-Quiver-Eq-Prop}, we
have $H = \prod_{t=0}^k \om{GL}(n_t,\C)$. Recall from the proof
of Theorem \ref{SemiSimpClass} that if we write $\gamma_t = \om{dim}Z_t\in \Z^{n+1}$ for $0\leq t\leq k$, then after possibly rearranging
the indices we have $\gamma_0 = (1,d-\beta) \in \Z^{n+1}$ and
$\gamma_t = (0,\alpha_t)\in \Z^{n+1}$, $1\leq t\leq k$, where
$\alpha_1,\alpha_2,\cdots,\alpha_t\in \Phi_i\subset \Phi\cap \tau^{\perp}$
is a base and $\beta = \sum_{t=1}^n n_t\alpha_t$ is the maximal root.

Let $z_j\in \mu_\C^{-1}(\lambda)\subset \om{Rep}(Q,\gamma_j)$ be the
point corresponding to $Z_j$. Then the complex $C_Q(x,x)$ decomposes
according to the decomposition $X=\sum_{t=0}^kn_tZ_t$, namely,
\[ C_Q(x,x) \cong \bigoplus_{t,s} \om{Hom}(\C^{n_s},\C^{n_t})\otimes C_Q(z_s,z_t)  .\]
The stabilizer $H=\prod_{t=0}^k \om{GL}(n_t,\C)$ only acts
on the first factors, i.e.,  
\[  (u_j)_{j}\cdot (f_{t,s}\otimes B_{t,s})_{t,s}
= (u_tf_{t,s}u_s^{-1}\otimes B_{t,s})_{t,s}  \]
for $(u_j)_j\in H$ and $f_{t,s}\otimes B_{t,s}\in \om{Hom}(\C^{n_s},\C^{n_t})\otimes C_Q(z_t,z_s)$.
Passing to cohomology we obtain
\begin{equation} \label{Space-Formula} 
H^1_Q(x,x) \cong \bigoplus_{s,t} \om{Hom}(\C^{n_s},\C^{n_t})\otimes H^1_Q(z_s,z_t)
\end{equation} 
and the action of $H$ is the same as described above. 
By Lemma \ref{Homology-Lemma} part $(3)$ and the fact that each $Z_t$
is a simple module we find
\begin{align} \label{Dimension-Formula}
 \om{dim}_\C H^1_Q(z_s,z_t) =& \; \om{dim}_C \om{Hom}_{\Pi^\lambda}(Z_s,Z_t)
+ \om{dim}_\C \om{Hom}_{\Pi^\lambda}(Z_t,Z_s)^*-(\gamma_s,\gamma_t) 
\nonumber \\  =& \; 2\delta_{st}-(\gamma_s,\gamma_t).
\end{align}  
Let $\wt{K}$ be the extended Dynkin graph associated with the root system 
$\Phi_i$. Specifically, the vertex set is $I = \{0,1,\cdots,k\}$ corresponding
to the roots $\alpha_0=-\beta,\alpha_1,\cdots,\alpha_t$ 
and a single edge connecting $s$ to $t$
if and only if $(\alpha_s,\alpha_t)=-1$. As $(\gamma_s,\gamma_t)=(\alpha_s,\alpha_t)$ for all $s,t$, we conclude by the dimension formula \eqref{Dimension-Formula} that
$H^1_Q(z_s,z_t)\cong \C$ precisely when $s\neq t$ and $s$ and $t$
are adjacent in $\wt{K}$ and $H^1_Q(z_s,z_t)=0$ otherwise.
The expression in \eqref{Space-Formula} then takes the form
\[ H^1_Q(x,x) \cong \bigoplus_{s\to t \; \mbox{ in } \wt{K}} \om{Hom}(\C^{n_s},\C^{n_t}) , \]
where each edge is repeated twice once with each orientation.
If the identifications $H^1_Q(z_s,z_t)\cong \C$ for $s$ and $t$ adjacent
in $\wt{K}$ are chosen appropriately,
the induced symplectic form is given by
\[ \omega ((f_{s,t})_{s,t},(g_{s,t})_{s,t})) = \sum_{s<t} \eps(s,t)(\om{tr}(f_{s,t}g_{t,s}) -\om{tr}(f_{t,s}g_{s,t}))  .\]
for some $\eps(s,t)=\pm 1$. If $s<t$ and there is an edge connecting $s$
to $t$, we specify the orientation of the edge by $s\to t$ if $\eps(s,t)=1$
and $t\to s$ if $\eps(t,s)=-1$. This gives rise to an extended Dynkin
quiver $Q'$ with minimal imaginary root $\delta'=(n_0=1,n_1,\cdots,n_k)$.
It is now clear from the above work that
$H^1_Q(x,x)\cong \om{Rep}(\ovl{Q'},\delta')$ as complex symplectic
$H\cong \prod_{t=0}^k \om{GL}(n_t,\C)=G_{\delta'}^c$ representations. 
\end{proof} 

To complete the proof of Theorem \ref{Singularity-Theorem} we will
need the following result.

\begin{lemma} \cite[Corollary~3.2]{Kronheimer89} Let $Q$ be an
extended Dynkin quiver with minimal imaginary root $\delta$.
Let $\Gamma\subset \om{SU}(2)$ be the finite subgroup associated
with the underlying unoriented graph of $Q$ under the MacKay correspondence.
Then there is a homeomorphism
\[ \MC{M}_0(Q,\delta) \cong \C^2/\Gamma  \]
that restricts to an isometry away from the singular point. In particular,
$\MC{M}_0^{\om{reg}}(Q,\delta) =\MC{M}_0(Q,\delta)-\{0\}$. 
\end{lemma} 

\begin{proof}[Proof of Theorem \ref{Singularity-Theorem}]
Let $Q$ be an extended Dynkin quiver with vertex set
$\{0,1,\cdots,n\}$ and minimal imaginary root $\delta =(1,d)\in \Z^{n+1}$,
where $d$ is the maximal positive root in the associated
root system $\Phi\subset \Z^n$ of type $ADE$. Let
$\lambda\in \C^{n+1}$ be a parameter satisfying $\lambda\cdot \delta=0$
and write $\lambda=(\lambda_1,\tau)\in \C\oplus \C^n$. Then by Theorem \ref{SemiSimpClass} there is a bijection between 
$\MC{M}_{(0,\lambda)}^{\om{sing}}(Q,\delta)$ and the components in the root space decomposition
\[ \Phi\cap \tau^\perp = \Phi_1\cup \Phi_2 \cup \cdots \cup \Phi_q . \]

Write $\MC{M}_{(0,\lambda)}^{\om{sing}}(Q,\delta) = \{x_1,\cdots,x_q\}$,
where $x_i$ corresponds to $\Phi_i$ for $1\leq i\leq q$. 
For each $1\leq i\leq q$, let $Q^{(i)}$ denote the extended Dynkin quiver associated with the root system $\Phi_i$ and let $\delta^{(i)}$ be
the associated minimal positive imaginary root.  
Then, according to Proposition \ref{Symplectic Slice} and Lemma \ref{Local-Structure-Lemma}, there is for each $1\leq i\leq q$ an open neighborhood
$U_i$ of $x_i\in \MC{M}_{(0,\lambda)}(Q,\delta)$, an open neighborhood
$V_i$ of $0\in \MC{M}_{0}(Q^{(i)},\delta^{(i)})$ and a biholomorphism $\rho_i\colon U_i\ra V_i$.
Importantly, since the category of complex manifolds is a full subcategory of the category of complex analytic spaces, this biholomorphism restricts to a biholomorphism $\rho_i\colon U_i^{\om{reg}} \cong V_i^{\om{reg}}$ of
complex manifolds. 

Let $\Gamma_i\subset \om{SU}(2)$ be the finite subgroup associated with
$Q^{(i)}$ under the McKay correspondence. 
By the above lemma there is for each $i$, $1\leq i\leq q$, a homeomorphism $\MC{M}_0(Q^{(i)},\delta^{(i)})\cong \C^2/\Gamma_i$ that restricts to an isometry away from the singular point. This map restricts to a homeomorphism
$\kappa_i\colon V_i\cong W_i\subset \C^2/\Gamma_i$ for some open neighborhood $W_i$ around $0$. By shrinking the $U_i$ and $V_i$ if necessary, we may assume
that $W_i = B_r(0)/\Gamma_i$ for some $r>0$ for each $1\leq i \leq q$.
The compositions $\phi_i \coloneqq \kappa_i\circ \rho_i\colon U_i\ra B_r(0)/\Gamma_i$ are then the required homeomorphisms. Indeed, for each $i$ both $\rho_i$
and $\kappa_i$ restrict to diffeomorphisms away from the singular point,
so we deduce that the restriction
\[ \phi_i=\kappa_i\circ \phi_i\colon  
\MC{M}_{(0,\lambda)}^{\om{reg}}(Q,\delta)\cap U_i
= U_i-\{x_i\} \cong (B_r(0)-\{0\})/\Gamma_i \]
is a diffeomorphism. This completes the proof.   
\end{proof}

\section{Configurations of Singularities and the Proof of Theorem \ref{Bordism-Theorem}}
Let $Q$ be an extended Dynkin quiver with vertex set $I=\{0,1,\cdots,n\}$
and minimal imaginary root $\delta\in \Z^{n+1}$. 
In this section we take up the question of what kind of configurations
of singularities that can occur in $\MC{M}_{(0,\lambda)}(Q,\delta)$ 
by varying the parameter $\lambda$. Assume that $\lambda\cdot \delta=0$ and
write $\lambda=(\lambda_1,\tau)\in \C\oplus \C^n$. Then according
to Theorem \ref{SemiSimpClass} and the local structure result in the
previous section, the configuration of singularities is uniquely determined
by the root space decomposition
\[ \Phi\cap \tau^\perp = \Phi_1\cup \cdots \cup \Phi_r,   \]
where $\Phi\subset \Z^n$ is the root system of type $ADE$ associated
with $Q$. The problem therefore reduces to determining the number and types
of root systems that can occur in the above root space decomposition.

Give $\C$ the total ordering determined by $z\leq w$ if and only
if either $\om{Re}(z)\leq \om{Re}(w)$ or $\om{Re}z=\om{Re}w$
and $\om{Im}z\leq \om{Im}w$. Note that this ordering is additive,
that is, $z\leq w\implies z+c\leq w+c$ for each $c\in \C$. 
We say that an element $\tau\in \C^n$ is dominant if $\tau_i\geq 0$ for each $i$. The value of this notion comes from the simple observation that if 
$\tau\in \C^n$ is dominant and $\theta\in \Z^n$ then $\tau\cdot \theta=0$ if and only if $\om{supp}(\theta)\cap \om{supp}(\tau) = \emptyset$. 

\begin{lemma} \label{Dynkin-Decomposition-Lemma} 
Let $K$ denote the Dynkin diagram associated with the
root system $\Phi\subset \Z^n$. Suppose $\tau\in \C^n$ is dominant and let
$J$ be the complement of $\om{supp}(\tau)$ in $\{1,2,\cdots,n\}$. Let
$K_J\subset K$ be the full subgraph of $K$ with vertex set
$J\subset \{1,2,\cdots,n\}$. Let
\[ K_J = K_1\sqcup K_2 \sqcup \cdots \sqcup K_r  \]
be the decomposition of $K_J$ into connected components. Then 
\[ \Phi\cap \tau^\perp = \Phi_1\cup \Phi_2 \cup \cdots \cup \Phi_r, \] 
where $\Phi_i$ is the $ADE$ root system associated with $K_i$ for
each $1\leq i\leq r$. 
\end{lemma}
\begin{proof} Note first that every connected subgraph of a Dynkin graph
of type $ADE$ is again a Dynkin graph of type $ADE$. Let $J_i$ be the
set of vertices for $K_i$ in the decomposition in the statement and put
$S_i = \{\eps_j : j\in J_i \}$. We claim that $S = \cup_i S_i$ is a base
for $\Phi\cap \tau^\perp$. Indeed, $S$ clearly consists of linearly independent
elements and every element $\alpha\in \Phi^+\cap \tau^\perp$
satisfies $\om{supp}(\alpha)\subset J$ so it can be written as a positive
linear integral combination of the elements of $S$. Then, as in the
proof of Lemma \ref{RootSpace-Decomp-Lemma}, the root space decomposition
\[ \Phi\cap \tau^\perp = \Phi_1\cup \cdots \Phi_r  \]
is obtained by decomposing $S$ into minimal pairwise orthogonal sets
$S = \cup_i S_i$ and letting $\Phi_i$ be the subsystem generated
by $S_i$. Importantly, this decomposition $S = \cup_i S_i$ is precisely
the decomposition introduced in the beginning. We conclude that
$\Phi_i$ is the root system associated with the Dynkin graph $K_i$
for each $1\leq i\leq r$. 
\end{proof}

For completeness we also show that the decomposition for an arbitrary
parameter $\tau$ can in fact be put in the above standard form.
Recall that the Weyl group associated with $\Phi$ is the finite group
$W\subset \om{Aut}_\Z(\Z^n)$ generated by the simple reflections
$s_i\colon \Z^n\ra \Z^n$ in the coordinate vectors $\eps_i$ for $1\leq i\leq n$. There is a unique action of $W$ on
$\C^n$ such that $(w\alpha)\cdot \tau = \alpha\cdot (w^{-1}\tau)$
for all $\alpha\in \Z^n$ and $\tau\in \C^n$. This is the complexification
of the dual action, where we identify $(\R^n)^*\cong \R^n$ using the
standard scalar product. 

The following lemma follows essentially
from the proof in  \cite[p.~51]{Humphreys78}, see also \cite[Lemma~7.2]{Crawley98}.

\begin{lemma} For every $\tau\in \C^n$ there exists $w\in W$ such
that $w\tau$ is dominant. \end{lemma}
\begin{proof} Write $\Phi = \Phi^+\cup \Phi^-$ and define
$\gamma = \frac12 \sum_{\alpha\in \Phi^+} \alpha$. By \cite[p.~50]{Humphreys78} one has $s_i(\gamma)=\gamma-\eps_i$ for each $1\leq i\leq n$.
Choose $w\in W$ such that $\gamma\cdot w\tau \geq \gamma\cdot w'\tau$ for
every $w'\in W$ with respect to the total ordering on $\C$. We claim that $\tau'\coloneqq w\cdot \tau$ is dominant.
Indeed, for each $1\leq i\leq n$ it holds true that
\[ \gamma\cdot \tau'\geq \gamma\cdot s_i\tau'
 = s_i\gamma \cdot \tau'=\gamma\cdot \tau'-\eps_i\cdot \tau'\]
or equivalently $\tau'_i=\eps_i\cdot \tau'\geq 0$. This shows that
$w\tau=\tau'$ is dominant.
\end{proof}  

\begin{proposition} \label{Configuration-Prop}
Let $K$ denote the Dynkin diagram associated with the
root system $\Phi\subset \Z^n$. Given $\tau\in \C^n$ let
\[ \Phi\cap \tau^\perp = \Phi_1\cup \cdots \cup \Phi_r \]
be the corresponding decomposition into $ADE$ root systems. Then there
exists a full subgraph $K'\subset K$ and a decomposition 
$K'=K_1\sqcup \cdots \sqcup K_r$ into connected components such that
$\Phi_i$ is isomorphic to the root system associated with $K_i$ for each
$i$.\end{proposition} 
\begin{proof} By the previous lemma there exists a Weyl transformation
$w\in W$ such that $w\tau\in \C^n$ is dominant. From the relation
$\tau \cdot \alpha = w\tau \cdot w\alpha$ we deduce that the isomorphism
$w\colon \Phi\ra \Phi$ restricts to an isomorphism $\tau^\perp\cap \Phi \ra (w\tau)^\perp\cap \Phi$. As this is an isomorphism of root systems,
it preserves the decomposition into irreducible components. The result
therefore follows from Lemma \ref{Dynkin-Decomposition-Lemma}
as $w\tau$ is dominant. \end{proof}

The final ingredient needed to complete the proof of Theorem \ref{Bordism-Theorem} is contained in the following proposition. We use the notation
$B_r(x)\subset \C^2$ and $\ovl{B}_r(x)\subset\C^2$ for the open and
closed ball, respectively, with center $x\in \C^2$ and radius $r$. 

\begin{proposition} \label{Structure-Infty-Prop}  
Let $Q$ be an extended Dynkin quiver with minimal imaginary root $\delta$.
Let $\Gamma\subset \om{SU}(2)$ be the finite subgroup associated with
the underlying extended Dynkin graph under the McKay correspondence.
Let $\lambda\in \C^{n+1}$ be a parameter with $\lambda\cdot \delta=0$.
Then there is an open subset $U\subset \MC{M}_{(0,\lambda)}^{\om{reg}}(Q,\delta)$ with compact complement in $\MC{M}_{(0,\lambda)}(Q,\delta)$ and a diffeomorphism $\phi\colon U\ra (\C^2-\ovl{B}_R(0))/\Gamma$. 
Moreover, $\phi^{-1}((\C^2-B_{R'}(0))/\Gamma)$ is closed in $\MC{M}_{(0,\lambda)}(Q,\delta)$ for each $R'>R$. \end{proposition}
\begin{remark} 
The final assertion is included to explicitly state that there
are no limit points in $\MC{M}_{(0,\lambda)}(Q,\delta)$ as
$x\in (\C^2-\ovl{B}_R(0))/\Gamma$ tends to $\infty$.
\end{remark}  
\begin{proof} Choose a parameter $\zeta \in \R^{n+1}$ satisfying
$\zeta\cdot \delta=0$ and $\zeta\cdot \theta\neq 0$ for each $\theta\in R_+(\delta)$ (defined in Theorem \ref{Smooth-Quotient-Theorem}) and
put $\xi = (0,\lambda)$ and $\wt{\xi}=(\zeta,\lambda)$. To simplify the
notation write
\[ \wt{X}=\MC{M}_{\wt{\xi}}(Q,\delta) \;\; \mbox{ and } \;\;
   X = \MC{M}_\xi(Q,\delta)  .\] 
Then according to Theorem \ref{Resolution-Sing-Theorem} 
there is a holomorphic map
$\pi\colon \wt{X}\ra X$ which is a resolution of singularities.
Furthermore, by Kronheimer's result mentioned in the introduction
\cite[Corollary~3.12]{Kronheimer89}, the smooth
$4$-dimensional hyper-K\"{a}hler manifold $\wt{X}$ is diffeomorphic
to the minimal resolution of the quotient singularity $\C^2/\Gamma$.
We may therefore assume that there is a continuous proper map
$\hat{\pi}\colon \wt{X}\ra \C^2/\Gamma$ that restricts to a diffeomorphism
$\hat{\pi}^{-1}((\C^2-\{0\})/\Gamma)\cong (\C^2-\{0\})/\Gamma$.
The situation is summarized in the following diagram
\[ \begin{tikzcd} X & \wt{X} \arrow{l}[swap]{\pi} \arrow{r}{\hat{\pi}} &
\C^2/\Gamma  \end{tikzcd} .\]
Since the open sets $\hat{\pi}^{-1}(B_R(0)/\Gamma)$ for $1<R<\infty$
cover $\wt{X}$ and $\pi^{-1}(X^{\om{sing}})$ is compact, there exists
an $R$ such that $\pi^{-1}(X^{\om{sing}})\subset \hat{\pi}^{-1}(B_R(0)/\Gamma)$.
Hence, 
\[ V\coloneqq \hat{\pi}^{-1}((\C^2-\ovl{B}_R(0))/\Gamma)\subset \pi^{-1}(X^{\om{reg}}),  \]
and as $\hat{\pi}$ is proper $X-V =\hat{\pi}^{-1}(\ovl{B}_R(0)/\Gamma)$ is
compact. The biholomorphism $\pi\colon \pi^{-1}(X^{\om{reg}})\cong X^{\om{reg}}$ therefore maps $V$ onto an open subset $U\subset X^{\om{reg}}$. The composition of the restrictions $\pi^{-1}\colon U\ra V$ and 
$\hat{\pi}\colon V\ra (\C^2-\ovl{B}_R(0))/\Gamma$ gives the required
diffeomorphism $\phi\colon U\cong (\C^2-\ovl{B}_R(0))/\Gamma$. Finally,
 \[ \phi^{-1}(\C^2-B_{R'}(0)))/\Gamma   = \pi(\hat{\pi}^{-1}(\C^2-B_{R'}(0))/\Gamma) \] 
is closed in $X$ for each $R'>R$ because $\hat{\pi}$ is continuous and
$\pi$ is a closed map (as it is proper and $X$ is locally compact Hausdorff).
\end{proof}

\begin{proof}[Proof of Theorem \ref{Bordism-Theorem}] Let
$\Gamma_0,\Gamma_1,\cdots,\Gamma_q\subset \om{SU}(2)$ be finite subgroups
and let $K_i$ denote the Dynkin diagram associated with $K_i$ for each
$0\leq i\leq q$. Assume that $K'\coloneqq  K_1\sqcup K_2 \sqcup \cdots \sqcup K_q$ can be realized
as a full subgraph of $K_0$. Identify the vertex set of $K_0$ with
$\{1,2,\cdots,n\}$ for some $n\in \N$ and let $J\subset \{1,\cdots,n\}$
be the vertices of the subgraph $K'$. Let $\Phi\subset \Z^n$ be the
root system associated with $K$ and specify $\tau\in \C^n$ by
$\tau_j = 1$ if $j\notin J$ and $\tau_j=0$ otherwise. Then $\tau$ is
dominant and $\om{supp}\tau$ is complementary to $J$. By Lemma \ref{Dynkin-Decomposition-Lemma} we have a root space decomposition
\begin{equation} \label{FinalProofEq}
\Phi\cap \tau^\perp = \Phi_1\cup \cdots \cup  \Phi_q,
\end{equation}
where $\Phi_i$ is the $ADE$ root system associated with
the Dynkin graph $K_i$ for each $1\leq i\leq q$. 

Let $Q$ be an extended Dynkin quiver with underlying extended Dynkin graph
corresponding to $\Gamma_0$ under the McKay correspondence (i.e., $\wt{K_0}$).
We identify the set of vertices with $\{0,1,\cdots,n\}$ such that the minimal
imaginary root is given by $(1,d)\in \Z^{n+1}$ where $d\in \Phi\subset \Z^n$
is the maximal positive root. Then $\lambda \coloneqq (-d\cdot \tau,\tau)\in \C^{n+1}$ satisfies $\lambda\cdot \delta=0$. Set $X\coloneqq \MC{M}_{(0,\lambda)}(Q,\delta)$. Then, according to Theorem \ref{Singularity-Theorem},
we may write $X^{\om{sing}}=\{x_1,x_2,\cdots,x_q\}$ and for each $1\leq i\leq q$ there is an open neighborhood $x_i\subset V_i\subset X$ and a homeomorphism
$\phi_i\colon V_i\ra B_r(0)/\Gamma_i$, for some fixed $r$ independent of $i$.
Furthermore, each $\phi_i$ restricts to a diffeomorphism away from the
singular point. Next, by Proposition \ref{Structure-Infty-Prop} there is an open subset $U'\subset X^{\om{reg}}$ with $X-U'$ compact and a diffeomorphism $\phi_0\colon U'\cong (\C^2-\ovl{B}_{R'}(0))/\Gamma_0$ for some $R'>0$. In addition, $\phi_0^{-1}((\C^2-B_{R}(0))/\Gamma)$ is closed in $X$ for each $R>R'$.

For part (i) we already know that $X^{\om{reg}}$ is a smooth hyper-K\"{a}hler
$4$-manifold. The space $X$ is connected by Lemma \ref{Connectivity}
and, in view of the above local models around the
singularities, it is clear that $X^{\om{reg}}=X-\{x_1,\cdots,x_q\}$ is
connected as well. 

For part (ii) and (iii) fix $R>R'$ and let $C\subset X$ be the closed 
subset $\phi^{-1}((\C^2-B_{R}(0))/\Gamma)$. Since $C\subset X^{\om{reg}}$
and $X$ is Hausdorff, we may assume after possibly shrinking the $V_i$ (and hence $r>0$) that the open sets $V_1,V_2,\cdots,V_q$ are pairwise disjoint
and that $V_i\cap C=\emptyset$ for each $i$. Put 
\[ U_0 \coloneqq \phi^{-1}((\C^2-\ovl{B}_R(0))/\Gamma)\subset X^{\om{reg}}
\; \mbox{ and } \; U_i \coloneqq V_i-\{x_i\}\subset X^{\om{reg}}, \; 
1\leq i\leq q. \]
Then the open subset $U_0,U_1,U_2,\cdots,U_q$ are pairwise disjoint, the complement of their union is compact in $X^{\om{reg}}$, and we
have diffeomorphisms $\phi_0\colon U_0\cong (\C^2-\ovl{B}_R(0))/\Gamma$
and $\phi_i\colon U_i\cong (B_r(0)-\{0\})/\Gamma$ for $1\leq i\leq q$.
We now decrease $r$ and increase $R$ slightly to ensure that each
$\phi_i$ extends over a slightly bigger open set for each $0\leq i\leq q$.  
The proof of part (ii) is completed by composing $\phi_0$ with the evident
diffeomorphism $(\C^2-\ovl{B}_R(0))/\Gamma \cong (R,\infty)\times S^3/\Gamma\cong (0,\infty)\times S^3/\Gamma$ and by composing $\phi_i$ with
the diffeomorphism
\[ (B_r(0)-\{0\})/\Gamma_i \cong (0,r)\times S^3/\Gamma_i \cong (0,\infty)\times S^3/\Gamma_i, \] 
where the final diffeomorphism includes a time reversal, for each $1\leq i\leq q$. Finally, $Y=X^{\om{reg}}-\cup_{i=0}^qU_i$ is compact a manifold with
boundary components $S^3/\Gamma_i$, $0\leq i\leq q$, 
because we arranged that $\phi_i$ actually
extends to a diffeomorphism $\phi_i'\colon U_i'\cong (-t_0,\infty)\times S^3/\Gamma_i$ for some $t_0>0$ for each $0\leq i\leq q$. This completes
the verification of part (iii) and hence the proof.   
\end{proof}

\printbibliography
\end{document}